\documentclass[11pt]{article}

\usepackage{amsmath}
\usepackage{amssymb,amsthm,bm,url,paralist,tabls,hyperref,amscd}
\usepackage{tikz,pgf}
\usetikzlibrary{matrix}

\newtheorem{thm}[equation]{Theorem}
\newtheorem{examp}[equation]{Example}
\newtheorem{prop}[equation]{Proposition}

\newtheorem{rem}[equation]{Remark}
\theoremstyle{remark}	  
\theoremstyle{definition} 

\newtheoremstyle{efronremark}% entire heading is optional argument []
{6pt}{6pt}{}{}{\itshape}{\quad}{ }{\thmnote{#3}}

\theoremstyle{efronremark}   \newtheorem*{eremark}{}

\numberwithin{equation}{section}
\usepackage{varioref}
\labelformat{section}{Section~#1}
\labelformat{subsection}{Section~#1}
\labelformat{subsubsection}{Section~#1}
\labelformat{figure}{Figure~#1}
\labelformat{table}{Table~#1}
\labelformat{defin}{Definition~#1}
\usepackage[letterpaper,margin=1in]{geometry}

\usepackage{setspace}
\title{Representations of the Rook-Brauer Algebra}
\author{Elise delMas\\  Department of Mathematics \\   Macalester College \\   St. Paul, MN 55105
\and
Tom Halverson\footnote{Supported in part by NSF Grant DMS-0800085.}
 \\  Department of Mathematics \\  Macalester College \\   St. Paul, MN 55105
}

\DeclareMathOperator{\End}{End}

\DeclareMathOperator{\Res}{Res}
\DeclareMathOperator{\Ker}{\mathbf{Ker}}

\date{June 20, 2012;  revised July 20, 2012}

\begin{document}

\newcommand{\SU}{\mathsf{SU}}
\newcommand{\SL}{\mathsf{SL}}
\newcommand{\CC}{\mathbb{C}}
\newcommand{\KK}{\mathbb{K}}
\newcommand{\ZZ}{\mathbb{Z}}
\newcommand{\V}{\mathsf{V}}
\newcommand{\VV}{\mathsf{V}}
\newcommand{\G}{\mathsf{G}}
\newcommand{\W}{\mathsf{W}}
\newcommand{\UU}{\mathsf{U}}
\newcommand{\U}{\mathsf{U}}
\renewcommand{\v}{\mathsf{v}}
\newcommand{\w}{\mathsf{w}}
\newcommand{\e}{\mathsf{e}}
\newcommand{\p}{\mathsf{p}}
\newcommand{\q}{\mathsf{q}}
\renewcommand{\b}{\mathsf{b}}
\newcommand{\f}{\mathsf{f}}
\newcommand{\Z}{\mathsf{Z}}
\newcommand{\M}{\mathsf{M}}
\newcommand{\RB}{\mathsf{RB}}
\newcommand{\RBc}{\mathcal{R\! B}}
\newcommand{\GL}{\mathsf{GL}}
\newcommand{\TL}{\mathsf{TL}}
\newcommand{\T}{\mathsf{T}}
\newcommand{\J}{\mathsf{J}}
\renewcommand{\Im}{\mathrm{Im}}
\newcommand{\rank}{\mathrm{rank}}
\newcommand{\mult}{\mathrm{mult}}
\def\ot{\otimes}
\def\id{{\rm id}}

\newcommand{\C}{\mathsf{C}}
\newcommand{\D}{\mathsf{D}}
\renewcommand{\T}{\mathsf{T}}
\renewcommand{\O}{\mathsf{O}}
\newcommand{\I}{\mathsf{I}}
\maketitle

\begin{abstract}
\noindent
 We study the  representation theory of  the rook-Brauer  algebra $\RB_k(x)$, which has a of Brauer diagrams that allow for the possibility of  missing edges.  The Brauer,  Temperley-Lieb,  Motzkin, rook monoid, and symmetric group algebras are subalgebras of $\RB_k(x)$.   We  prove that $\RB_k(n+1)$  is the centralizer  of the orthogonal group $\O_n(\CC)$  on tensor space and that $\RB_k(n+1)$ and $\O_n(\CC)$ are in Schur-Weyl duality.  When $x \in \CC$ is chosen so that $\RB_k(x)$ is semisimple, we use its Bratteli diagram  to explicitly construct a complete set of irreducible representations for $\RB_k(x)$.  
\end{abstract}

\begin{eremark}[Keywords:] Brauer algebra,  rook monoid, orthogonal group, Schur-Weyl duality, seminormal representation.
\end{eremark}

\begin{section}{Introduction}

For $k \ge 0$ and $x \in \CC$, the rook-Brauer algebra $\RB_k(x)$ is an associative algebra with 1 over $\CC$ with a basis given by rook-Brauer diagrams --- these are Brauer diagrams in which we may remove edges --- and a multiplication given by diagram concatenation.   The rook-Brauer algebra contains the symmetric group algebra, the rook monoid algebra \cite{So}, the Brauer algebra \cite{Br}, the Temperley-Lieb algebra \cite{TL}, the Motzkin algebra \cite{BH}, and the planar rook monoid algebra \cite{FHH} as subalgebras.

In this paper, we define an action of $\RB_k(n+1)$ on tensor space $\V^{\otimes k}$ where $\V$ has dimension $n+1$ with basis $v_0, v_1, \ldots, v_n$, and we show that this action is faithful when $n \ge k$.  We then consider $\V = \V(0) \oplus \V(1)$ where $\V(0) = \CC\text{-span}\{v_0\}$ is the trivial module for the orthogonal group $\O_n(\CC)$, and  $\V(1) = \CC\text{-span}\{ v_1, \ldots, v_n\}$ is the defining $n$-dimensional module for $\O_n(\CC)$.   Then  $\V^{\otimes k}$ is a $k$-fold tensor product module for $\O_n(\CC)$.  We prove that the actions of $\RB_k(n+1)$ and $\O_n(\CC)$ commute on $\V^{\otimes k}$ and that,
$$
\RB_k(n+1) \cong \End_{\O_n(\CC)} (\V^{\otimes k}), \qquad \hbox{ when $n \ge k$.}
$$
Thus, $\RB_k(n+1)$ and $\O_n(\CC)$ centralize one another  and are in Schur-Weyl duality on $\V^{\otimes k}$.

Using this duality, we are able to show that for generic values of $x \in \CC$, $\RB_k(x)$ is semisimple.  Double centralizer theory and the combinatorics of the irreducible $\O_n(\CC)$ modules which appear in $\V^{\otimes k}$ tell us that the irreducible $\RB_k(n+1)$ modules are indexed by integer partitions $\lambda \vdash r$ for $0 \le r \le k$.   Furthemore, dimension of these irreducible modules are given by paths in the Bratteli diagram, which are closely related to the vacillating tableaux studied in \cite{HR} and \cite{CDDSY}.

In Section 4 of this paper we explicitly construct the irreducible $\RB_k(x)$-modules for $x \in \CC$ such that $\RB_k(x)$ is semisimple.  These modules have a basis indexed by  paths in the Bratteli diagram, and we give an action of the generators on this basis.  Our action is a generalization of Young's seminormal action of the symmetric group on standard Young tableaux.  We use the analogous seminormal representations of the Brauer algebra \cite{Nz}, \cite{LR} and of the rook monoid algebra \cite{Ha} in our construction.

The rook-Brauer algebra was first introduced in \cite{GrW}, and it is also studied in \cite{Mz}. The monoid of rook-Brauer diagrams (called the partial Brauer monoid) is examined in  \cite{KM}, where it is given a presentation on generators and relations.  Representations of the rook-Brauer algebra are studied  independently (and simultaneously with this work) by Martin and Mazorchuk \cite{MM} using different methods.   Martin and Mazorchuk also  derive the Schur-Weyl duality between $\O_n(\CC)$ and $\RB_k(n)$. They define an action of the generators on tensor space and verify that the defining  relations are satisfied. Our approach is different in that we  give the action of an arbitrary rook-Brauer diagram on tensor space and use the diagram calculus to prove that it is a representation (see Proposition \ref{IsARep}).  Furthermore, Martin and Mazorchuk \cite{MM} construct  Specht modules for the (integral) rook Brauer algebra, which correspond to a complete set of irreducible $\RB_k(x)$ modules over $\CC$ in the semisimple case.  These are different from, but isomorphic to,  the seminormal representations given in Section 4 of this paper.

The Schur-Weyl duality results in this paper form the bulk of the work in the first author's 2012 honors thesis \cite{dM}. The ideas behind this   paper began long ago in conversations between the second author, Georgia Benkart, and Cheryl Grood. It was to become a paper, ``BMWISH Algebras" (Birman-Murakami-Wenzl-Iwahori-Solomon-Hecke Algebras), which was  never written.   We thank the anonymous referee for several helpful suggestions.

\end{section}

\begin{section}{The Rook Brauer Algebra $\RB_k(x)$}

\begin{subsection}{Rook-Brauer diagrams}

A \emph{rook-Brauer $k$-diagram} is a graph consisting of two rows each with $k$ vertices such that each vertex is connected to at most one other vertex by an edge. We let $\RBc_k$ denote the set of all rook-Brauer $k$-diagrams, so that for example,
$$
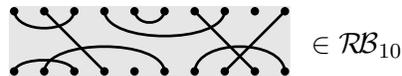
\begin{array}{c}
\begin{tikzpicture}[scale=.4,line width=1pt] 
\foreach \i in {1,...,10} 
{ \path (\i,1) coordinate (T\i); \path (\i,-1) coordinate (B\i); } 
%%%%%%%
\filldraw[fill= black!10,draw=black!10,line width=4pt]  (T1) -- (T10) -- (B10) -- (B1) -- (T1);
\draw[black] (T7) -- (B9);
\draw[black] (T10) -- (B8);
\draw[black] (T2) -- (B4);
\draw[black] (T1) .. controls +(.1,-.75) and +(-.1,-.75) .. (T3) ;
\draw[black] (T4) .. controls +(.1,-1.1) and +(-.1,-1.1) .. (T8) ;
\draw[black] (T5) .. controls +(.1,-.5) and +(-.1,-.5) .. (T6) ;
\draw[black] (B1) .. controls +(.1,.5) and +(-.1,.5) .. (B3) ;
\draw[black] (B7) .. controls +(.1,1.1) and +(-.1,1.1) .. (B10) ;
\draw[black] (B2) .. controls +(.1,1.1) and +(-.1,1.1) .. (B6) ;
%%%%%%%
\foreach \i in {1,...,10} 
{ \fill (T\i) circle (4pt); \fill (B\i) circle (4pt); } 
\end{tikzpicture}\end{array} \in \RBc_{10}
$$
and
 $$
\RBc_{2} = 
\left\{
\begin{array}{c}
\begin{tikzpicture}[scale=.4,line width=1pt] 
\foreach \i in {1,...,2} 
{ \path (\i,.5) coordinate (T\i); \path (\i,-.5) coordinate (B\i); } 
\filldraw[fill= black!10,draw=black!10,line width=4pt]  (T1) -- (T2) -- (B2) -- (B1) -- (T1);
\draw[black] (T1) -- (B1);
\draw[black] (T2) -- (B2);
\foreach \i in {1,...,2} 
{ \fill (T\i) circle (4pt); \fill (B\i) circle (4pt); } 
\end{tikzpicture},\ 
\begin{tikzpicture}[scale=.4,line width=1pt] 
\foreach \i in {1,...,2} 
{ \path (\i,1) coordinate (T\i); \path (\i,0) coordinate (B\i); } 
\filldraw[fill= black!10,draw=black!10,line width=4pt]  (T1) -- (T2) -- (B2) -- (B1) -- (T1);
\draw[black] (T1) -- (B2);
\draw[black] (T2) -- (B1);
\foreach \i in {1,...,2} 
{ \fill (T\i) circle (4pt); \fill (B\i) circle (4pt); } 
\end{tikzpicture},\ 
\begin{tikzpicture}[scale=.4,line width=1pt] 
\foreach \i in {1,...,2} 
{ \path (\i,1) coordinate (T\i); \path (\i,0) coordinate (B\i); } 
\filldraw[fill= black!10,draw=black!10,line width=4pt]  (T1) -- (T2) -- (B2) -- (B1) -- (T1);
\draw[black] (T1) .. controls +(.1,-.35) and +(-.1,-.35) .. (T2) ;
\draw[black] (B1) .. controls +(.1,.35) and +(-.1,.35) .. (B2) ;
\foreach \i in {1,...,2} 
{ \fill (T\i) circle (4pt); \fill (B\i) circle (4pt); } 
\end{tikzpicture},\ 
\begin{tikzpicture}[scale=.4,line width=1pt] 
\foreach \i in {1,...,2} 
{ \path (\i,1) coordinate (T\i); \path (\i,0) coordinate (B\i); } 
\filldraw[fill= black!10,draw=black!10,line width=4pt]  (T1) -- (T2) -- (B2) -- (B1) -- (T1);
\draw[black] (T1) -- (B1);
\foreach \i in {1,...,2} 
{ \fill (T\i) circle (4pt); \fill (B\i) circle (4pt); } 
\end{tikzpicture},\ 
\begin{tikzpicture}[scale=.4,line width=1pt] 
\foreach \i in {1,...,2} 
{ \path (\i,1) coordinate (T\i); \path (\i,0) coordinate (B\i); } 
\filldraw[fill= black!10,draw=black!10,line width=4pt]  (T1) -- (T2) -- (B2) -- (B1) -- (T1);
\draw[black] (T2) -- (B2);
\foreach \i in {1,...,2} 
{ \fill (T\i) circle (4pt); \fill (B\i) circle (4pt); } 
\end{tikzpicture},\ 
\begin{tikzpicture}[scale=.4,line width=1pt] 
\foreach \i in {1,...,2} 
{ \path (\i,1) coordinate (T\i); \path (\i,0) coordinate (B\i); } 
\filldraw[fill= black!10,draw=black!10,line width=4pt]  (T1) -- (T2) -- (B2) -- (B1) -- (T1);
\draw[black] (T1) -- (B2);
\foreach \i in {1,...,2} 
{ \fill (T\i) circle (4pt); \fill (B\i) circle (4pt); } 
\end{tikzpicture},\ 
\begin{tikzpicture}[scale=.4,line width=1pt] 
\foreach \i in {1,...,2} 
{ \path (\i,1) coordinate (T\i); \path (\i,0) coordinate (B\i); } 
\filldraw[fill= black!10,draw=black!10,line width=4pt]  (T1) -- (T2) -- (B2) -- (B1) -- (T1);
\draw[black] (T2) -- (B1);
\foreach \i in {1,...,2} 
{ \fill (T\i) circle (4pt); \fill (B\i) circle (4pt); } 
\end{tikzpicture},\ 
\begin{tikzpicture}[scale=.4,line width=1pt] 
\foreach \i in {1,...,2} 
{ \path (\i,1) coordinate (T\i); \path (\i,0) coordinate (B\i); } 
\filldraw[fill= black!10,draw=black!10,line width=4pt]  (T1) -- (T2) -- (B2) -- (B1) -- (T1);
\draw[black] (B1) .. controls +(.1,.35) and +(-.1,.35) .. (B2) ;
\foreach \i in {1,...,2} 
{ \fill (T\i) circle (4pt); \fill (B\i) circle (4pt); } 
\end{tikzpicture},\ 
\begin{tikzpicture}[scale=.4,line width=1pt] 
\foreach \i in {1,...,2} 
{ \path (\i,1) coordinate (T\i); \path (\i,0) coordinate (B\i); } 
\filldraw[fill= black!10,draw=black!10,line width=4pt]  (T1) -- (T2) -- (B2) -- (B1) -- (T1);
\draw[black] (T1) .. controls +(.1,-.35) and +(-.1,-.35) .. (T2) ;
\foreach \i in {1,...,2} 
{ \fill (T\i) circle (4pt); \fill (B\i) circle (4pt); } 
\end{tikzpicture},\ 
\begin{tikzpicture}[scale=.4,line width=1pt] 
\foreach \i in {1,...,2} 
{ \path (\i,1) coordinate (T\i); \path (\i,0) coordinate (B\i); } 
\filldraw[fill= black!10,draw=black!10,line width=4pt]  (T1) -- (T2) -- (B2) -- (B1) -- (T1);
\foreach \i in {1,...,2} 
{ \fill (T\i) circle (4pt); \fill (B\i) circle (4pt); } 
\end{tikzpicture}\end{array} \right\}.
$$
The rook-Brauer $k$-diagrams correspond to  the set of partial matchings of $\{1, 2, \ldots, 2 k\}$.  Rook-Brauer diagrams consist of \emph{vertical edges} which connect a vertex in the top row to a vertex in the bottom row, \emph{horizontal edges} which connect vertices in the same row, and \emph{isolated vertices}  which are not incident to an edge.  We can count the number of rook-Brauer diagrams of size $k$ according to the number $\ell$ of edges in the diagram. First choose, in $\binom{2 k} {2 \ell}$ ways, the $2 \ell$ vertices to be paired by an edge, and then connect these vertices in $(2 \ell)!! = (2 \ell -1) (2 \ell -3) \cdots 3 \cdot 1$ ways.  Thus
\begin{equation}
\vert \RBc_k \vert = \sum_{\ell = 0}^k \binom{2 k} {2 \ell} (2 \ell)!!, 
\end{equation}
so that the first few values are $1, 2, 10, 76, 764, 9496, 140152, 2390480, 46206736, \ldots$.

\end{subsection}

\begin{subsection}{The rook-Brauer algebras $\RB_k(x,y)$ and $\RB_k(x)$}

We  follow \cite{Mz} and first define a two parameter rook-Brauer algebra $\RB_k(x,y)$.  We then specialize to the one-parameter algebra  $\RB_k(x) = \RB_k(x,1)$.
Assume $\KK$ is a commutative ring with 1 and that $x$ and $y$ are elements of $\KK$.   Set  $\RB_0(x,y) = \KK 1$, and for $k \geq 1$,  let $\RB_k(x,y)$ be the free $\KK$-module with basis the rook-Brauer $k$-diagrams $\RBc_k$.  We multiply two rook-Brauer $k$-diagrams $d_1$ and $d_2$ as follows. Place $d_1$ above $d_2$ and identify the bottom-row vertices in $d_1$ with the corresponding top-row vertices in $d_2$. The product is 
\begin{equation}\label{multiplication}
d_1d_2 = x^{\kappa(d_1,d_2)} y^{\varepsilon(d_1,d_2)} d_3
\end{equation}
 where $d_3$ is the diagram consisting of the horizontal edges from the top row of $d_1$, the horizontal edges from the bottom row of $d_2$, and the vertical edges that propagate from the bottom of $d_2$ to the top of $d_1$,  $\kappa(d_1, d_2)$ is the number of loops that arise in the middle row, and $\varepsilon(d_1,d_2)$ is the number of isolated vertices that are left in the middle row. For example, if 
$$
d_1 = \begin{array}{c}
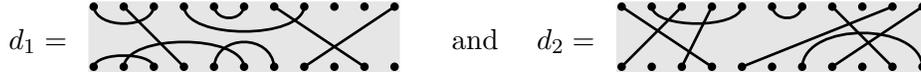
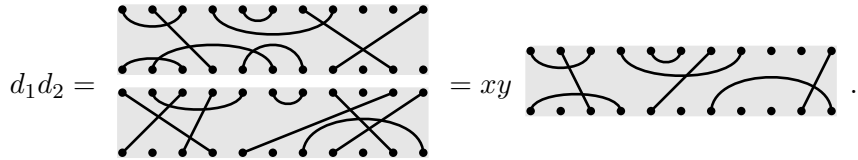
\begin{tikzpicture}[scale=.4,line width=1pt] 
\foreach \i in {1,...,11} 
{ \path (\i,1) coordinate (T\i); \path (\i,-1) coordinate (B\i); } 
%%%%%%%
\filldraw[fill= black!10,draw=black!10,line width=4pt]  (T1) -- (T11) -- (B11) -- (B1) -- (T1);
\draw[black] (T7) -- (B10);
\draw[black] (T11) -- (B8);
\draw[black] (T2) -- (B4);
\draw[black] (T1) .. controls +(.1,-.75) and +(-.1,-.75) .. (T3) ;
\draw[black] (T4) .. controls +(.1,-1.1) and +(-.1,-1.1) .. (T8) ;
\draw[black] (T5) .. controls +(.1,-.5) and +(-.1,-.5) .. (T6) ;
\draw[black] (B1) .. controls +(.1,.5) and +(-.1,.5) .. (B3) ;
\draw[black] (B5) .. controls +(.1,1.1) and +(-.1,1.1) .. (B7) ;
\draw[black] (B2) .. controls +(.1,1.1) and +(-.1,1.1) .. (B6) ;
%%%%%%%
\foreach \i in {1,...,11} 
{ \fill (T\i) circle (4pt); \fill (B\i) circle (4pt); } 
\end{tikzpicture}\end{array} \quad\hbox{ and } \quad
d_2 = \begin{array}{c}
\begin{tikzpicture}[scale=.4,line width=1pt] 
\foreach \i in {1,...,11} 
{ \path (\i,1) coordinate (T\i); \path (\i,-1) coordinate (B\i); } 
%%%%%%%
\filldraw[fill= black!10,draw=black!10,line width=4pt]  (T1) -- (T11) -- (B11) -- (B1) -- (T1);
\draw[black] (T10) -- (B5);
\draw[black] (T11) -- (B8);
\draw[black] (T1) -- (B4);
\draw[black] (T3) -- (B1);
\draw[black] (T4) -- (B3);
\draw[black] (T8) -- (B10);
\draw[black] (T2) .. controls +(.1,-.75) and +(-.1,-.75) .. (T5) ;
\draw[black] (T6) .. controls +(.1,-.5) and +(-.1,-.5) .. (T7) ;
\draw[black] (B7) .. controls +(.1,1.5) and +(-.1,1.5) .. (B11) ;
%%%%%%%
\foreach \i in {1,...,11} 
{ \fill (T\i) circle (4pt); \fill (B\i) circle (4pt); } 
\end{tikzpicture}\end{array}
$$
then
\begin{equation*}
d_1 d_2 = 
 \begin{array}{c}
\begin{tikzpicture}[scale=.4,line width=1pt] 
\foreach \i in {1,...,11} 
{ \path (\i,1) coordinate (T\i); \path (\i,-1) coordinate (B\i); } 
%%%%%%%
\filldraw[fill= black!10,draw=black!10,line width=4pt]  (T1) -- (T11) -- (B11) -- (B1) -- (T1);
\draw[black] (T7) -- (B10);
\draw[black] (T11) -- (B8);
\draw[black] (T2) -- (B4);
\draw[black] (T1) .. controls +(.1,-.75) and +(-.1,-.75) .. (T3) ;
\draw[black] (T4) .. controls +(.1,-1.1) and +(-.1,-1.1) .. (T8) ;
\draw[black] (T5) .. controls +(.1,-.5) and +(-.1,-.5) .. (T6) ;
\draw[black] (B1) .. controls +(.1,.5) and +(-.1,.5) .. (B3) ;
\draw[black] (B5) .. controls +(.1,1.1) and +(-.1,1.1) .. (B7) ;
\draw[black] (B2) .. controls +(.1,1.1) and +(-.1,1.1) .. (B6) ;
%%%%%%%
\foreach \i in {1,...,11} 
{ \fill (T\i) circle (4pt); \fill (B\i) circle (4pt); } 
\end{tikzpicture} \\
\begin{tikzpicture}[scale=.4,line width=1pt] 
\foreach \i in {1,...,10} 
{ \path (\i,1) coordinate (T\i); \path (\i,-1) coordinate (B\i); } 
%%%%%%%
\filldraw[fill= black!10,draw=black!10,line width=4pt]  (T1) -- (T11) -- (B11) -- (B1) -- (T1);
\draw[black] (T10) -- (B5);
\draw[black] (T11) -- (B8);
\draw[black] (T1) -- (B4);
\draw[black] (T3) -- (B1);
\draw[black] (T4) -- (B3);
\draw[black] (T8) -- (B10);
\draw[black] (T2) .. controls +(.1,-.75) and +(-.1,-.75) .. (T5) ;
\draw[black] (T6) .. controls +(.1,-.5) and +(-.1,-.5) .. (T7) ;
\draw[black] (B7) .. controls +(.1,1.5) and +(-.1,1.5) .. (B11) ;
%%%%%%%
\foreach \i in {1,...,11} 
{ \fill (T\i) circle (4pt); \fill (B\i) circle (4pt); } 
\end{tikzpicture}\end{array}
= x y
\begin{array}{c}
\begin{tikzpicture}[scale=.4,line width=1pt] 
\foreach \i in {1,...,11} 
{ \path (\i,1) coordinate (T\i); \path (\i,-1) coordinate (B\i); } 
%%%%%%%
\filldraw[fill= black!10,draw=black!10,line width=4pt]  (T1) -- (T11) -- (B11) -- (B1) -- (T1);
\draw[black] (T7) -- (B5);
\draw[black] (T11) -- (B10);
\draw[black] (T2) -- (B3);
\draw[black] (T1) .. controls +(.1,-.75) and +(-.1,-.75) .. (T3) ;
\draw[black] (T4) .. controls +(.1,-1.1) and +(-.1,-1.1) .. (T8) ;
\draw[black] (T5) .. controls +(.1,-.5) and +(-.1,-.5) .. (T6) ;
\draw[black] (B1) .. controls +(.1,.75) and +(-.1,.75) .. (B4) ;
\draw[black] (B7) .. controls +(.1,1.5) and +(-.1,1.5) .. (B11) ;
%%%%%%%
\foreach \i in {1,...,11} 
{ \fill (T\i) circle (4pt); \fill (B\i) circle (4pt); } 
\end{tikzpicture}\end{array}.
\end{equation*}
Diagram multiplication makes $\RB_k(x,y)$ into an associative algebra with identity element
$
\mathbf{1}_k = 
\begin{array}{c}
\begin{tikzpicture}[scale=.4,line width=1pt] 
\foreach \i in {1,...,6} 
{ \path (\i,.5) coordinate (T\i); \path (\i,-.5) coordinate (B\i); } 
\filldraw[fill= black!10,draw=black!10,line width=4pt]  (T1) -- (T6) -- (B6) -- (B1) -- (T1);
\draw[black] (T1) -- (B1);
\draw[black] (T2) -- (B2);
\draw[black] (T3) -- (B3);
\draw[black] (T5) -- (B5);
\draw[black] (T6) -- (B6);
\foreach \i in {1,...,3} 
{ \fill (T\i) circle (4pt); \fill (B\i) circle (4pt); } 
\foreach \i in {5,...,6} 
{ \fill (T\i) circle (4pt); \fill (B\i) circle (4pt); } 
\draw (T4) node {$\cdots$};
\draw (B4) node {$\cdots$};
\end{tikzpicture}\end{array}.
$

Under this multiplication, two vertical
edges can become a single horizontal edge as in the example above, 
and a vertical edge can contract to a vertex.  
The {\em rank} of a diagram $\rank(d)$ is the number of vertical edges in the diagram.  Therefore,  
$
\rank( d_1 d_2) \le \min( \rank(d_1), \rank(d_2) ).
$ 
For $0 \le r \le k$, we let $\J_r(x,y) \subseteq \RB_k(x)$ be the $\KK$-span of the rook-Brauer $k$-diagrams of rank \emph{less than or equal} to $r$.   Then $\J_r(x,y)$ is a two-sided ideal in $\RB_k(x)$ and we have the tower of ideals,
$\J_0(x,y) \subseteq \J_1(x,y) \subseteq \J_2(x,y) \subseteq \cdots \subseteq \J_k (x,y) = \RB_k(x,y).$

Under multiplication in the partition algebra $\mathsf{P}_k(x)$ defined in \cite{Ma} and \cite{Jo} isolated vertices and inner loops are treated the same.  Thus, it is possible to view the rook-Brauer algebra $\RB_k(x,y)$ as a subalgebra of the partition algebra $\mathsf{P}_k(x)$ by setting the two parameters equal to one-another; specifically, $\RB_k(x,x) \subseteq \mathsf{P}_k(x)$.  For the Schur-Weyl duality between the rook Brauer algebra and the orthogonal group in Section 3, we want to specialize the second parameter to $y = 1$.  For this reason, we consider the one-parameter rook-Brauer algebra
\begin{equation}
\RB_k(x) = \RB_k(x,1).
\end{equation}

\end{subsection}

\begin{subsection}{Subalgebras}
A number of important diagram algebras live in $\RB_k(x)$ as subalgebras (we use $\RB_k(x)$ instead of $\RB_k(x,y)$ since each subalgebra either has no isolated vertices or requires that $y =1$).  Rook-Brauer $k$-diagrams having $k$ edges are Brauer diagrams, and their $\KK$-span is the Brauer algebra $\mathsf{B}_k(x)$.  The set of diagrams of rank $k$ is the symmetric group $\mathsf{S}_k$ and their $\CC$-span is the symmetric group algebra $\KK \mathsf{S}_k$.  The set of diagrams with no horizontal edges is the rook monoid $\mathsf{R}_k$ and their span is the rook monoid algebra $\KK \mathsf{R}_k$.  Finally, a rook-Brauer diagram is \emph{planar} if can be drawn within the rectangle  formed by its edges without any edge crossings.  The planar Brauer diagrams span the Temperley-Lieb algebra $\mathsf{TL}_k(x)$, the planar rook-monoid diagrams span the planar rook monoid algebra $\mathsf{PR}_k$, and the planar rook-Brauer diagrams span the Motzkin algebra $\mathsf{M}_k(x)$.   An example of a diagram from each of these subalgebras follows:
$$
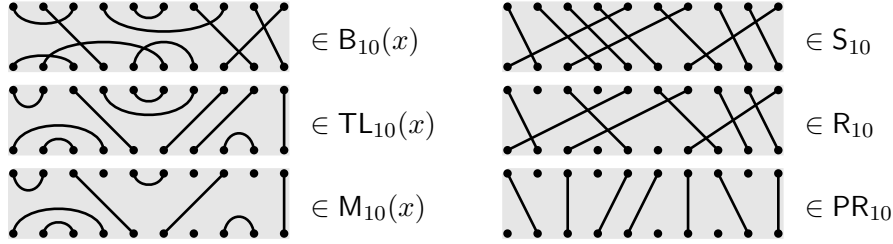
\begin{array}{l c l}
\begin{array}{c}
\begin{tikzpicture}[scale=.4,line width=1pt] 
\foreach \i in {1,...,10}  { \path (\i,1) coordinate (T\i); \path (\i,-1) coordinate (B\i); } 
\filldraw[fill= black!10,draw=black!10,line width=4pt]  (T1) -- (T10) -- (B10) -- (B1) -- (T1);
\draw[black] (T7) -- (B9);
\draw[black] (T10) -- (B8);
\draw[black] (T9) -- (B10);
\draw[black] (T2) -- (B4);
\draw[black] (T1) .. controls +(.1,-.75) and +(-.1,-.75) .. (T3) ;
\draw[black] (T4) .. controls +(.1,-1.1) and +(-.1,-1.1) .. (T8) ;
\draw[black] (T5) .. controls +(.1,-.5) and +(-.1,-.5) .. (T6) ;
\draw[black] (B1) .. controls +(.1,.5) and +(-.1,.5) .. (B3) ;
\draw[black] (B5) .. controls +(.1,1.1) and +(-.1,1.1) .. (B7) ;
\draw[black] (B2) .. controls +(.1,1.1) and +(-.1,1.1) .. (B6) ;
\foreach \i in {1,...,10}  { \fill (T\i) circle (4pt); \fill (B\i) circle (4pt); } 
\end{tikzpicture}\end{array}
 \in \mathsf{B}_{10}(x) &  & %%%%%%%%%%%%%%%%%%%%%%%%%%%%%%%%%%%%%%%
 \begin{array}{c}
 \begin{tikzpicture}[scale=.4,line width=1pt] 
\foreach \i in {1,...,10}  { \path (\i,1) coordinate (T\i); \path (\i,-1) coordinate (B\i); } 
\filldraw[fill= black!10,draw=black!10,line width=4pt]  (T1) -- (T10) -- (B10) -- (B1) -- (T1);
\draw[black] (T1) -- (B2);\
\draw[black] (T2) -- (B4);
\draw[black] (T3) -- (B5);
\draw[black] (T4) -- (B6);
\draw[black] (T5) -- (B1);
\draw[black] (T6) -- (B8);
\draw[black] (T7) -- (B3);
\draw[black] (T8) -- (B9);
\draw[black] (T9) -- (B10);
\draw[black] (T10) -- (B7);
\foreach \i in {1,...,10}  { \fill (T\i) circle (4pt); \fill (B\i) circle (4pt); } 
\end{tikzpicture}\end{array}
 \in \mathsf{S}_{10} 
 \\
 %%%%%%%%%%%%%%%%%%%%%%%%%%%%%%%%%%%%%%%
 %%%%%%%%%%%%%%%%%%%%%%%%%%%%%%%%%%%%%%%
 \begin{array}{c}
 \begin{tikzpicture}[scale=.4,line width=1pt] 
\foreach \i in {1,...,10}  { \path (\i,1) coordinate (T\i); \path (\i,-1) coordinate (B\i); } 
\filldraw[fill= black!10,draw=black!10,line width=4pt]  (T1) -- (T10) -- (B10) -- (B1) -- (T1);
\draw[black] (T8) -- (B6);
\draw[black] (T10) -- (B10);
\draw[black] (T9) -- (B7);
\draw[black] (T3) -- (B5);
\draw[black] (T1) .. controls +(.1,-.75) and +(-.1,-.75) .. (T2) ;
\draw[black] (T4) .. controls +(.1,-1.1) and +(-.1,-1.1) .. (T7) ;
\draw[black] (T5) .. controls +(.1,-.5) and +(-.1,-.5) .. (T6) ;
\draw[black] (B2) .. controls +(.1,.5) and +(-.1,.5) .. (B3) ;
\draw[black] (B8) .. controls +(.1,.75) and +(-.1,.75) .. (B9) ;
\draw[black] (B1) .. controls +(.1,1.1) and +(-.1,1.1) .. (B4) ;
\foreach \i in {1,...,10}  { \fill (T\i) circle (4pt); \fill (B\i) circle (4pt); } 
\end{tikzpicture}\end{array}
 \in \mathsf{TL}_{10}(x)
&  & %%%%%%%%%%%%%%%%%%%%%%%%%%%%%%%%%%%%%%%
 \begin{array}{c}
\begin{tikzpicture}[scale=.4,line width=1pt] 
\foreach \i in {1,...,10}  { \path (\i,1) coordinate (T\i); \path (\i,-1) coordinate (B\i); } 
\filldraw[fill= black!10,draw=black!10,line width=4pt]  (T1) -- (T10) -- (B10) -- (B1) -- (T1);
\draw[black] (T1) -- (B2);
\draw[black] (T3) -- (B5);
\draw[black] (T5) -- (B1);
\draw[black] (T6) -- (B8);
\draw[black] (T7) -- (B3);
\draw[black] (T8) -- (B9);
\draw[black] (T9) -- (B10);
\draw[black] (T10) -- (B7);
\foreach \i in {1,...,10}  { \fill (T\i) circle (4pt); \fill (B\i) circle (4pt); } 
\end{tikzpicture}\end{array}
 \in \mathsf{R}_{10} \\
  %%%%%%%%%%%%%%%%%%%%%%%%%%%%%%%%%%%%%%%
 %%%%%%%%%%%%%%%%%%%%%%%%%%%%%%%%%%%%%%%
 \begin{array}{c}
\begin{tikzpicture}[scale=.4,line width=1pt] 
\foreach \i in {1,...,10}  { \path (\i,1) coordinate (T\i); \path (\i,-1) coordinate (B\i); } 
\filldraw[fill= black!10,draw=black!10,line width=4pt]  (T1) -- (T10) -- (B10) -- (B1) -- (T1);
\draw[black] (T8) -- (B6);
\draw[black] (T10) -- (B10);
\draw[black] (T3) -- (B5);
\draw[black] (T1) .. controls +(.1,-.75) and +(-.1,-.75) .. (T2) ;
\draw[black] (T5) .. controls +(.1,-.5) and +(-.1,-.5) .. (T6) ;
\draw[black] (B2) .. controls +(.1,.5) and +(-.1,.5) .. (B3) ;
\draw[black] (B8) .. controls +(.1,.75) and +(-.1,.75) .. (B9) ;
\draw[black] (B1) .. controls +(.1,1.1) and +(-.1,1.1) .. (B4) ;
\foreach \i in {1,...,10}  { \fill (T\i) circle (4pt); \fill (B\i) circle (4pt); } 
\end{tikzpicture}\end{array}
 \in \mathsf{M}_{10}(x)  &  & %%%%%%%%%%%%%%%%%%%%%%%%%%%%%%%%%%%%%%%
 \begin{array}{c}
\begin{tikzpicture}[scale=.4,line width=1pt] 
\foreach \i in {1,...,10}  { \path (\i,1) coordinate (T\i); \path (\i,-1) coordinate (B\i); } 
\filldraw[fill= black!10,draw=black!10,line width=4pt]  (T1) -- (T10) -- (B10) -- (B1) -- (T1);
\draw[black] (T1) -- (B2);
\draw[black] (T3) -- (B3);
\draw[black] (T5) -- (B4);
\draw[black] (T6) -- (B5);
\draw[black] (T7) -- (B7);
\draw[black] (T8) -- (B9);
\draw[black] (T10) -- (B10);
\foreach \i in {1,...,10}  { \fill (T\i) circle (4pt); \fill (B\i) circle (4pt); } 
\end{tikzpicture}\end{array}
 \in \mathsf{PR}_{10} \\
\end{array}
$$
\end{subsection}

\begin{subsection}{Generators and relations}

For $1 \le i \le k-1$ define the following special elements of $\RB_k(x)$
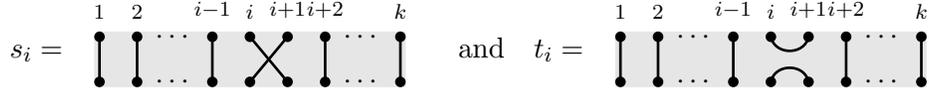
$$
s_i = \begin{array}{c}
\begin{tikzpicture}[scale=.5,line width=1pt] 
\foreach \i in {1,...,9} 
{ \path (\i,1) coordinate (T\i); \path (\i,-.2) coordinate (B\i); } 
%%%%%%%
\filldraw[fill= black!10,draw=black!10,line width=4pt]  (T1) -- (T9) -- (B9) -- (B1) -- (T1);
\draw[black] (T1) -- (B1);
\draw[black] (T2) -- (B2);
\draw[black] (T4) -- (B4);
\draw[black] (T5) -- (B6);
\draw[black] (T6) -- (B5);
\draw[black] (T7) -- (B7);
\draw[black] (T9) -- (B9);
%%%%%%%
\foreach \i in {1,...,2}  { \fill (T\i) circle (4pt); \fill (B\i) circle (4pt); } 
\foreach \i in {4,...,7}  { \fill (T\i) circle (4pt); \fill (B\i) circle (4pt); } 
{ \fill (T9) circle (4pt); \fill (B9) circle (4pt); } 
\draw (T3) node {$\cdots$};\draw (B3) node {$\cdots$};
\draw (T8) node {$\cdots$};\draw (B8) node {$\cdots$};
\draw  (T1)  node[above=0.1cm]{$\scriptstyle{1}$};
\draw  (T2)  node[above=0.1cm]{$\scriptstyle{2}$};
\draw  (T4)  node[above=0.1cm]{$\scriptstyle{i-1}$};
\draw  (T5)  node[above=0.1cm]{$\scriptstyle{i}$};
\draw  (T6)  node[above=0.1cm]{$\scriptstyle{i+1}$};
\draw  (T7)  node[above=0.1cm]{$\scriptstyle{i+2}$};
\draw  (T9)  node[above=0.1cm]{$\scriptstyle{k}$};
\end{tikzpicture}
\end{array}
\quad\hbox{and}\quad
t_i = \begin{array}{c}
\begin{tikzpicture}[scale=.5,line width=1pt] 
\foreach \i in {1,...,9} 
{ \path (\i,1) coordinate (T\i); \path (\i,-.2) coordinate (B\i); } 
%%%%%%%
\filldraw[fill= black!10,draw=black!10,line width=4pt]  (T1) -- (T9) -- (B9) -- (B1) -- (T1);
\draw[black] (T1) -- (B1);
\draw[black] (T2) -- (B2);
\draw[black] (T4) -- (B4);
\draw[black] (T5) .. controls +(.1,-.5) and +(-.1,-.5) .. (T6) ;
\draw[black] (B5) .. controls +(.1,.5) and +(-.1,.5) .. (B6) ;
\draw[black] (T7) -- (B7);
\draw[black] (T9) -- (B9);
%%%%%%%
\foreach \i in {1,...,2}  { \fill (T\i) circle (4pt); \fill (B\i) circle (4pt); } 
\foreach \i in {4,...,7}  { \fill (T\i) circle (4pt); \fill (B\i) circle (4pt); } 
{ \fill (T9) circle (4pt); \fill (B9) circle (4pt); } 
\draw (T3) node {$\cdots$};\draw (B3) node {$\cdots$};
\draw (T8) node {$\cdots$};\draw (B8) node {$\cdots$};
\draw  (T1)  node[above=0.1cm]{$\scriptstyle{1}$};
\draw  (T2)  node[above=0.1cm]{$\scriptstyle{2}$};
\draw  (T4)  node[above=0.1cm]{$\scriptstyle{i-1}$};
\draw  (T5)  node[above=0.1cm]{$\scriptstyle{i}$};
\draw  (T6)  node[above=0.1cm]{$\scriptstyle{i+1}$};
\draw  (T7)  node[above=0.1cm]{$\scriptstyle{i+2}$};
\draw  (T9)  node[above=0.1cm]{$\scriptstyle{k}$};
\end{tikzpicture}
\end{array}
$$
and for $1 \le i \le k$ define
$$
p_i =\begin{array}{c}
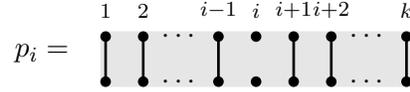
\begin{tikzpicture}[scale=.5,line width=1pt] 
\foreach \i in {1,...,9} 
{ \path (\i,1) coordinate (T\i); \path (\i,-.2) coordinate (B\i); } 
%%%%%%%
\filldraw[fill= black!10,draw=black!10,line width=4pt]  (T1) -- (T9) -- (B9) -- (B1) -- (T1);
\draw[black] (T1) -- (B1);
\draw[black] (T2) -- (B2);
\draw[black] (T4) -- (B4);
\draw[black] (T6) -- (B6);
\draw[black] (T7) -- (B7);
\draw[black] (T9) -- (B9);
%%%%%%%
\foreach \i in {1,...,2}  { \fill (T\i) circle (4pt); \fill (B\i) circle (4pt); } 
\foreach \i in {4,...,7}  { \fill (T\i) circle (4pt); \fill (B\i) circle (4pt); } 
{ \fill (T9) circle (4pt); \fill (B9) circle (4pt); } 
\draw (T3) node {$\cdots$};\draw (B3) node {$\cdots$};
\draw (T8) node {$\cdots$};\draw (B8) node {$\cdots$};
\draw  (T1)  node[above=0.1cm]{$\scriptstyle{1}$};
\draw  (T2)  node[above=0.1cm]{$\scriptstyle{2}$};
\draw  (T4)  node[above=0.1cm]{$\scriptstyle{i-1}$};
\draw  (T5)  node[above=0.1cm]{$\scriptstyle{i}$};
\draw  (T6)  node[above=0.1cm]{$\scriptstyle{i+1}$};
\draw  (T7)  node[above=0.1cm]{$\scriptstyle{i+2}$};
\draw  (T9)  node[above=0.1cm]{$\scriptstyle{k}$};
\end{tikzpicture}
\end{array}
$$
It is easy to verify that these diagrams generate all of $\RB_k(x)$.  
In fact, $\RB_k(x)$ can be generated by $p_1, t_1$ and $s_1, \ldots, s_{n-1}$.  Furthermore,  our 
subalgebras are generated as follows,
$$
\begin{array}{lcl}
\mathsf{B}_k(x) = \langle s_1, \ldots, s_{k-1}, t_1, \ldots, t_{k-1} \rangle  & &  
\KK \mathsf{R}_k =  \langle p_1, \ldots, p_k, s_1, \ldots, s_{k-1} \rangle \\
\KK \mathsf{S}_k = \langle s_1, \ldots, s_{k-1} \rangle & & 
\KK \mathsf{PR}_k =  \langle p_1, \ldots, p_k, \ell_1, \ldots, \ell_{k-1},r_1, \ldots, r_{k-1}  \rangle   \\
\mathsf{TL}_k(x) = \langle t_1, \ldots, t_{k-1} \rangle & &
\mathsf{M}_k(x) = \langle t_1, \ldots, t_{k-1}, \ell_1, \ldots, \ell_{k-1}, r_1, \ldots, r_{k-1} \rangle \\
\end{array}
$$
where, for each $1 \le i \le k-1$, we let $\ell_i = s_i p_i$ and $r_i = p_i s_i$.

A presentation of $\RB_k(x)$ is given in \cite{KM} which we describe here.  The symmetric group $\mathsf{S}_k\subseteq \RB_k(x)$ is generated by $s_i$ for $1\leq i \leq k-1$ subject to the following relations:
\begin{equation}\label{Snrels}
\begin{array}{lll}
\text{(a)} & s_i^2 =1, & 1 \le i \le k-1, \\
\text{(b)} & s_is_j = s_js_i, & |i-j| > 1,\\
\text{(c)} & s_i s_{i+1} s_i = s_{i+1} s_i s_{i+1}, & 1 \le i \le k-2.
\end{array}\hskip.4truein
\end{equation}
The Brauer algebra $\mathsf{B}_k(x) \subseteq \RB_k(x)$ contains $\mathsf{S}_k$ and is generated by $s_i$ and $t_i$ subject to relations \eqref{Snrels} and the following:
\begin{equation}\label{Bnrels}
\begin{array}{lll}
\text{(a)} & t_i^2 = (x+1) t_i, &1 \le i \le k-1, \\
\text{(b)} &  t_i t_j = t_jt_i \hbox{ and } t_i s_j = s_jt_i,   & |i-j| >1,\\
\text{(c)} &  t_i s_i = s_i t_i = t_i, & 1 \le i \le k-1, \\
\text{(d)} & t_it_{i\pm1} t_i = t_i, & 1 \le i \le k-1,\\
\text{(e)} & s_it_{i\pm1}t_i = s_{i\pm1}t_i \hbox{ and }  t_i t_{i\pm1} s_i = t_is_{i\pm1},& 1 \le i \le k-1, \\
\end{array}\hskip.5truein
\end{equation}
The rook monoid $\mathsf{R}_k \subseteq \RB_k(x)$ is generated by $s_i$ and $p_i$ subject to relations  \eqref{Snrels}  and the following relations:
\begin{equation}\label{Rnrels}
\begin{array}{lll}
\text{(a)} &p_i^2 = p_i, & 1 \le i \le k, \\
\text{(b)} &p_ip_j = p_jp_i, & i\neq j,\\
\text{(c)} &s_ip_i = p_{i+1}s_i, \quad & 1 \le i \le k-1,\\
\text{(d)} &s_ip_j = p_js_i, &  |i-j| > 1,\\
\text{(e)} &p_is_ip_i = p_ip_{i+1}, & 1 \le i \le k-1.
\end{array}\hskip.7truein
\end{equation}
Finally, the rook-Brauer algebra $\RB_k(x)$ is generated by $s_i, t_i,$ and $p_i$ subject to relations \eqref{Snrels}, \eqref{Bnrels}, and \eqref{Rnrels} along with the following relations:
\begin{equation}\label{RBnrels}
\begin{array}{lll}
\text{(a)} &t_ip_j = p_jt_i, & |i-j| > 1;\\
\text{(b)} &t_ip_i = t_i p_{i+1} = t_ip_ip_{i+1}, & 1 \le i \le k-1, \\
\text{(c)} &p_it_i = p_{i+1}t_i = p_ip_{i+1}t_i, \quad &  1 \le i \le k-1, \\
\text{(d)} &t_ip_it_i = t_i, &  1 \le i \le k-1, \\
\text{(e)} &p_it_ip_i = p_ip_{i+1}, &  1 \le i \le k-1.
\end{array}
\end{equation}
\end{subsection}
\end{section}

\begin{section}{Schur-Weyl Duality}

\begin{subsection}{Tensor powers of orthogonal group modules}

Let $\GL_n(\CC)$ denote the general linear group of $n \times n$ invertible matrices with entries from the complex numbers $\CC$, and let $\O_n(\CC) = \{\ g \in \GL_n(\CC)\ \vert\ g g^t =  I\ \}$ be the subgroup of orthogonal matrices, where $t$ denotes matrix transpose and $I$ is the $n \times n$ identity matrix.  We refer to \cite{GW} or \cite{LR} for standard results on the representation theory of the orthogonal group.
Let $\V(0)= \CC$ be the 1-dimensional trivial $\GL_n(\CC)$ module and let $\V(1)=\CC^n$ be the defining $n$-dimensional $\GL_n(\CC)$ module on which  $\GL_n(\CC)$ acts as $n \times n$ matrices. Specifically, let $v_0$ be a basis for $\V(0)$ and let $v_1, \ldots, v_n$ be a basis for $\V(1)$, then for $g \in \GL_n(\CC)$ we have $g v_0 = v_0$ and $g v_j = \sum_{j=1 }^n g_{ij} v_i,$ where $g_{ij}$ is the $ij$-entry of $g$.

Then  $\V = \V(0) \oplus \V(1)$ is an $(n+1)$-dimensional $\GL_n(\CC)$ module with basis $v_0, v_1, \ldots, v_n$. Consider the $k$-fold tensor product module
\begin{equation}
\V^{\otimes k} = \CC\hbox{-span}\left\{\ v_{i_1} \ot \cdots  \ot v_{i_k} \ \vert \ i_j \in \{0, \ldots, n\} \ \right\},
\end{equation}
which has dimension $(n+1)^k$ and a basis consisting of simple tensors of the form $v_{i_1} \ot \cdots  \ot v_{i_k}$.
An element $g \in \GL_n(\CC)$ acts on a simple tensor by the diagonal action
\begin{equation}
g(v_{i_1} \ot \cdots  \ot v_{i_k}) = (g v_{i_1}) \ot \cdots  \ot (gv_{i_k}),
\end{equation}
which extends linearly to make $\V^{\otimes k}$ a $\GL_n(\CC)$ module. 

The irreducible polynomial representations of $\O_n(\CC)$ are of the form $\V^{\lambda}$ where $\lambda = (\lambda_1, \lambda_2, \ldots, \lambda_\ell)$ is a partition with $\lambda_1' + \lambda_2' \le n$.  In this notation $\V(0) = \V^\emptyset$ and $\V(1) = \V^{(1)} = \V^\square$.  The Clebsch-Gordon formulas give the following tensor product formulas,
\begin{equation}
\V^\lambda \otimes \V^\emptyset = \V^\lambda \qquad\hbox{and}\qquad
\V^\lambda \otimes \V^{(1)} = \bigoplus_{\mu = \lambda \pm \square} \V^\mu,\\
\end{equation}
and thus,
\begin{equation}\label{eq:TensorProductRule}
\V^\lambda \otimes \V = \bigoplus_{\mu} \V^\mu,
\end{equation}
where the sum is over all partitions $\mu$ such that either $\mu = \lambda$ or $\mu = \lambda \pm \square$.  By convention, we define $\V^{\otimes 0} = \V^\emptyset$.   

Let  $\lambda \vdash r$ denote the fact that $\lambda$ is an integer partition of $r$, and for $k \ge 0$, define
\begin{equation}
\Lambda_k =  \left\{ \  \lambda \vdash r \ \vert\ 0 \le r \le k \ \right\}.
\end{equation}
Define the \emph{Bratteli diagram} $\mathcal{B}$ of $\O_n(\CC)$ acting on $\V^{\otimes k}$ to be the infinite rooted lattice with vertices on level $k$ labeled by the partitions in $\Lambda_k$ and an edge between $\mu \in \Lambda_{k-1}$ and  $\lambda \in \Lambda_k$ if and only if one of the following three conditions hold:
$\lambda = \mu, \lambda = \mu + \square,$ or
$\lambda = \mu - \square.$
Thus the first 4 rows of $\mathcal{B}$ are given by
$$
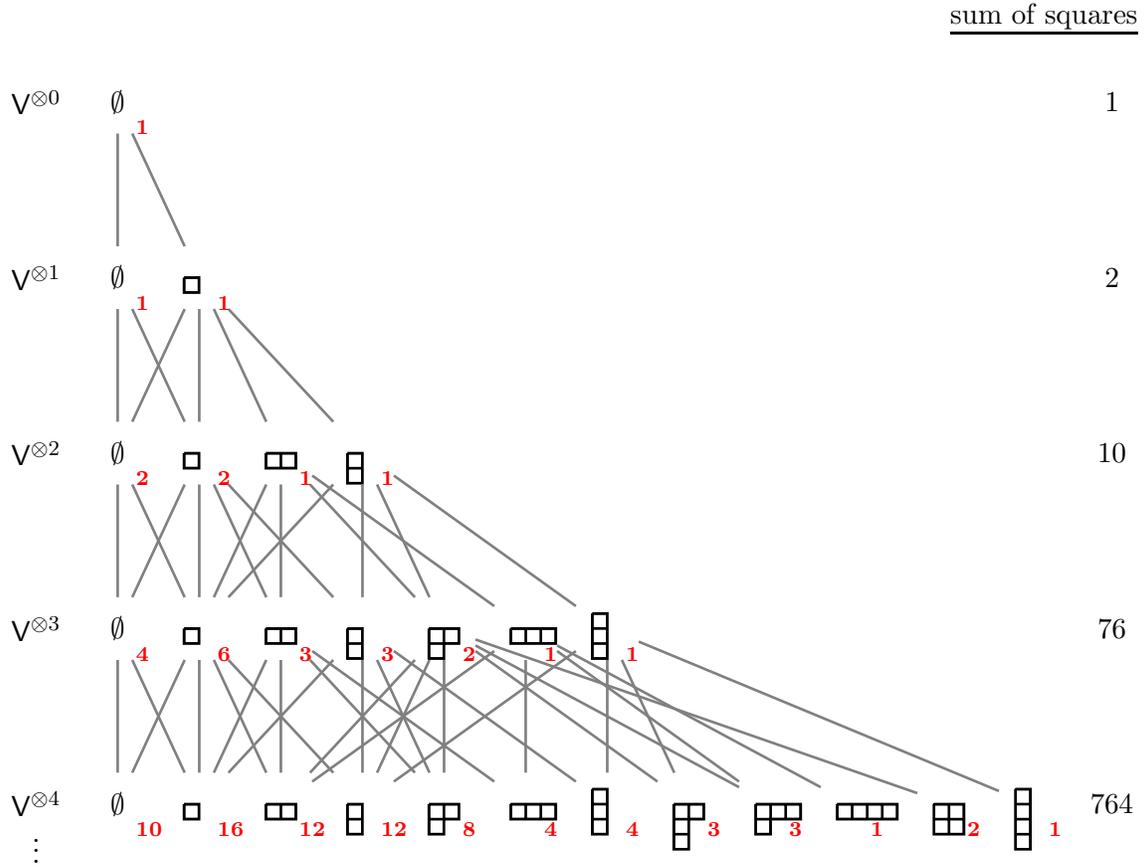
\begin{tikzpicture}[line width=1pt,scale=0.025]
  
\def\partA{
\node{$\emptyset$};
       }
    
\def\partB{
\node{};
 \draw[clip,scale=0.2] (-1,0) -| (0,0) -| (0,-1) -| (-1,-1) -| (-1,0); 
    \draw[scale=0.2] (-5,-5) grid (5,5);
    }
    
\def\partC{
\node{};
    \draw[clip,scale=0.2] (-1,0) -| (1,0) -| (1,-1) -| (-1,-1) -|  (-1,0); 
    \draw[scale=0.2] (-5,-5) grid (5,5);
    }
    
\def\partD{
\node{};
    \draw[clip,scale=0.2] (-1,0) -| (0,0) -| (0,-2) -| (-1,-2) -|  (-1,0); 
    \draw[scale=0.2] (-5,-5) grid (5,5);
    }
    
\def\partE{
\node{};
    \draw[clip,scale=0.2] (-1,0) -| (1,0) -| (1,-1) -| (0,-1) -|  (0,-2) -| (-1,-2) -| (-1,0); 
    \draw[scale=0.2] (-5,-5) grid (5,5);
    }
    
\def\partF{
\node{};
	\draw[clip,scale=.2] (-1,0) -| (2,0) -| (2,-1) -| (-1,-1) -| (-1,0);
	\draw[scale=.2] (-5,-5) grid (5,5);
}

\def\partG{
\node{};
	\draw[clip,scale=.2] (-1,1) -| (0,1) -| (0,-2) -| (-1,-2) -| (-1,1);
	\draw[scale=.2] (-5,-5) grid (5,5);
}
\def\partH{
\node{};
    \draw[clip,scale=0.2] (-1,0) -| (1,0) -| (1,-1)-| (0,-1) -|  (0,-3) -| (-1,-3) -| (-1,0); 
    \draw[scale=0.2] (-5,-5) grid (5,5);
}
    
\def\partI{
\node{};
    \draw[clip,scale=0.2] (-1,0) -| (1,0) -| (1,-2) -| (-1,-2) -|   (-1,0); 
    \draw[scale=0.2] (-5,-5) grid (5,5);
}
    
 \def\partJ{
 \node{};
 	\draw[clip,scale=0.2] (-1,0) -| (2,0) -| (2,-1) -| (0,-1) -| (0,-2) -| (-1,-2) -| (-1,0);
	\draw[scale=0.2] (-5,-5) grid (5,5);
 }

\def\partK{
\node{};
	\draw[clip,scale=0.2] (-1,0) -| (3,0) -| (3,-1) -| (-1,-1) -| (-1,0);
	\draw[scale=.2] (-5,-5) grid (5,5);
}

\def\partL{
\node{};
	\draw[clip,scale=0.2] (-1,1) -| (0,1) -| (0,-3) -| (-1,-3) -| (-1,1);
	\draw[scale=.2] (-5,-5) grid (5,5);

}
    
\matrix (m) [matrix of nodes, column sep=.25cm,row sep = 1.5 cm,ampersand replacement=\&,
nodes={draw=white, % General options for all nodes
      line width=1pt,
      anchor=center, 
       rounded corners,
      minimum width=.8cm, minimum height=.8cm
    }]    
 { 
\node{}; \& \partA  \& \node{};  \&\node{};\&\node{};\&\node{};\&\node{};\&\node{};\&\node{};\& \node{};\& \node{};\& \node{};\& \node{};\& \node{};\\
\node{}; \& \partA  \& \partB \& \node{}; \&\node{}; \&\node{};\&\node{};\&\node{};\&\node{};\& \node{};\& \node{};\& \node{};\& \node{};\& \node{};\\ 
\node{}; \& \partA  \& \partB \& \partC \& \partD\&\node{};\&\node{};\&\node{};\&\node{};\& \node{};\& \node{};\& \node{};\& \node{};\& \node{};\\ 
\node{}; \& \partA  \& \partB \& \partC \& \partD \& \partE \& \partF \& \partG \&\node{};\& \node{};\& \node{};\& \node{};\& \node{};\& \node{};\\
\node{}; \& \partA  \& \partB \& \partC \& \partD \& \partE \& \partF \& \partG \&\partH \& \partJ \& \partK \& \partI \& \partL\& \node{};\\
 };
 
 \draw (m-1-1) node {$\V^{\otimes 0}$};
\path (m-1-2) edge[black!50] (m-2-2);
\path (m-1-2) edge[black!50] (m-2-3);

\draw (m-2-1) node {$\V^{\otimes 1}$};

\path (m-2-2) edge[black!50] (m-3-2);
\path (m-2-2) edge[black!50] (m-3-3);
\path (m-2-3) edge[black!50] (m-3-3);
\path (m-2-3) edge[black!50] (m-3-2);
\path (m-2-3) edge[black!50] (m-3-4);
\path (m-2-3) edge[black!50] (m-3-5);

\draw (m-3-1) node {$\V^{\otimes 2}$};

\path (m-3-2) edge[black!50] (m-4-2);
\path (m-3-2) edge[black!50] (m-4-3);
\path (m-3-3) edge[black!50] (m-4-3);
\path (m-3-3) edge[black!50] (m-4-2);
\path (m-3-3) edge[black!50] (m-4-4);
\path (m-3-3) edge[black!50] (m-4-5);
\path (m-3-4) edge[black!50] (m-4-3);
\path (m-3-4) edge[black!50] (m-4-4);
\path (m-3-4) edge[black!50] (m-4-7);
\path (m-3-4) edge[black!50] (m-4-6);
\path (m-3-5) edge[black!50] (m-4-3);
\path (m-3-5) edge[black!50] (m-4-8);
\path (m-3-5) edge[black!50] (m-4-5);
\path (m-3-5) edge[black!50] (m-4-6);

\draw (m-4-1) node {$\V^{\otimes 3}$};

\path (m-4-2) edge[black!50] (m-5-2);
\path (m-4-2) edge[black!50] (m-5-3);

\path (m-4-3) edge[black!50] (m-5-3);
\path (m-4-3) edge[black!50] (m-5-2);
\path (m-4-3) edge[black!50] (m-5-4);
\path (m-4-3) edge[black!50] (m-5-5);

\path (m-4-4) edge[black!50] (m-5-3);
\path (m-4-4) edge[black!50] (m-5-4);
\path (m-4-4) edge[black!50] (m-5-7);
\path (m-4-4) edge[black!50] (m-5-6);

\path (m-4-5) edge[black!50] (m-5-3);
\path (m-4-5) edge[black!50] (m-5-8);
\path (m-4-5) edge[black!50] (m-5-5);
\path (m-4-5) edge[black!50] (m-5-6);

\path (m-4-6) edge[black!50] (m-5-4);
\path (m-4-6) edge[black!50] (m-5-5);
\path (m-4-6) edge[black!50] (m-5-6);
\path (m-4-6) edge[black!50] (m-5-9);
\path (m-4-6) edge[black!50] (m-5-10);
\path (m-4-6) edge[black!50] (m-5-12);

\path (m-4-7) edge[black!50] (m-5-4);
\path (m-4-7) edge[black!50] (m-5-7);
\path (m-4-7) edge[black!50] (m-5-10);
\path (m-4-7) edge[black!50] (m-5-11);

\path (m-4-8) edge[black!50] (m-5-5);
\path (m-4-8) edge[black!50] (m-5-8);
\path (m-4-8) edge[black!50] (m-5-9);
\path (m-4-8) edge[black!50] (m-5-13);

\draw (m-5-1) node {$\V^{\otimes 4}$};

\draw (250,235) node[black] {\underline{sum of squares}};
\draw (m-1-2) node[below right=.1cm, red] {$\scriptstyle\bf 1$};

\draw (m-2-2) node[below right=.1cm, red] {$\scriptstyle\bf 1$};
\draw (m-2-3) node[below right=.1cm, red] {$\scriptstyle\bf 1$};

\draw (m-3-2) node[below right=.1cm, red] {$\scriptstyle\bf 2$};
\draw (m-3-3) node[below right=.1cm, red] {$\scriptstyle\bf 2$};
\draw (m-3-4) node[below right=.1cm, red] {$\scriptstyle\bf 1$};
\draw (m-3-5) node[below right=.1cm, red] {$\scriptstyle\bf 1$};

\draw (m-4-2) node[below right=.1cm, red] {$\scriptstyle\bf 4$};
\draw (m-4-3) node[below right=.1cm, red] {$\scriptstyle\bf 6$};
\draw (m-4-4) node[below right=.1cm, red] {$\scriptstyle\bf 3$};
\draw (m-4-5) node[below right=.1cm, red] {$\scriptstyle\bf 3$};
\draw (m-4-6) node[below right=.1cm, red] {$\scriptstyle\bf 2$};
\draw (m-4-7) node[below right=.1cm, red] {$\scriptstyle\bf 1$};
\draw (m-4-8) node[below right=.1cm, red] {$\scriptstyle\bf 1$};

\draw (m-5-1) node[below = .1cm] {$\mathbf{\vdots}$};
\draw (m-5-2) node[below right=.1cm, red] {$\scriptstyle\bf 10$};
\draw (m-5-3) node[below right=.1cm, red] {$\scriptstyle\bf 16$};
\draw (m-5-4) node[below right=.1cm, red] {$\scriptstyle\bf 12$};
\draw (m-5-5) node[below right=.1cm, red] {$\scriptstyle\bf 12$};
\draw (m-5-6) node[below right=.1cm, red] {$\scriptstyle\bf 8$};
\draw (m-5-7) node[below right=.1cm, red] {$\scriptstyle\bf 4$};
\draw (m-5-8) node[below right=.1cm, red] {$\scriptstyle\bf 4$};
\draw (m-5-9) node[below right=.1cm, red] {$\scriptstyle\bf 3$};
\draw (m-5-10) node[below right=.1cm, red] {$\scriptstyle\bf 3$};
\draw (m-5-11) node[below right=.1cm, red] {$\scriptstyle\bf 1$};
\draw (m-5-12) node[below right=.1cm, red] {$\scriptstyle\bf 2$};
\draw (m-5-13) node[below right=.1cm, red] {$\scriptstyle\bf 1$};

\draw (m-1-14) node[black] {1};
\draw (m-2-14) node[black] {2};
\draw (m-3-14) node[black] {10};
\draw (m-4-14) node[black] {76};
\draw (m-5-14) node[black] {764};
 
 \end{tikzpicture}
$$
By this construction, we see that $\Lambda_k$ indexes the irreducible $\O_n(\CC)$ modules which appear in $\V^{\otimes k}$. Furthermore, if we let
 $m_k^\lambda$ denote the multiplicity of $\V^\lambda$ in $\V^{\otimes k}$, then  by induction on the tensor product rule \eqref{eq:TensorProductRule}, we see that
\begin{equation}\label{eq:DimensionPaths}
m_k^\lambda = (\hbox{the number of paths of length $k$   in $\mathcal{B}$ from $\emptyset \in \Lambda_0$ to $\lambda \in \Lambda_k$}).
\end{equation} 
\end{subsection}

\begin{subsection}{Centralizer of $\O_n(\CC)$ on $\V^{\otimes k}$}
 
Let $\C_k = \End_{\O_n(\CC)}(\VV^{\otimes k})$ be the centralizer of $\O_n(\CC)$ acting on $\VV^{\otimes k}$, so that  
\begin{equation}
\C_k = \left\{\  \phi \in \End(\VV^{\otimes k}) \  \big\vert \   \phi (y w) = y\phi(w) \hbox{ for all $y \in \O_n(\CC)$, $w \in \VV^{\otimes k}$} \ \right\}.
\end{equation}
Then by the classical double-centralizer theory (see for example \cite{HR} or \cite[Secs.\ 3B and 68]{CR}), we know the
following.  Let $m_{k,\lambda} = \mult_{\V^{\otimes k}}(\V^\lambda)$ denote the multiplicity of $\V^\lambda$ in $\V^{\otimes k}$. Let $\dim(\V^\lambda) = P_\lambda(n)$ be the El-Samra-King polynomials (see Section 4). We have,
\begin{itemize}
\item $\C_k$ is a semisimple associative $\CC$-algebra with irreducible representations labeled by $\Lambda_k$. We let
$
\left\{\  \M_k^{\lambda}  \ \vert \ \lambda \vdash r, \ 0 \le r \le k \right\}
$
denote the set of irreducible $\C_k$-modules.

\item $\dim(\M_k^\lambda) = m_k^\lambda$.

\item We can naturally embed the algebras  $\C_0 \subseteq \C_1 \subseteq \C_2 \cdots $,  and the edges from level $k+1$
to level $k$ in $\mathcal{B}$ represent the restriction rule for $\C_k \subseteq \C_{k+1}$.  Therefore, $\mathcal{B}$ is
the Bratteli diagram for the tower of semisimple algebras $\C_k$. 

\item The tensor space $\VV^{\otimes k}$ decomposes as \begin{equation}
\begin{array}{rll}
\VV^{\otimes k} & \cong \displaystyle{\bigoplus_{r = 0}^k \bigoplus_{\lambda\vdash r}}\,\, m_k^\lambda\, \V^\lambda & \hbox{ as an $\O_n(\CC)$-module}, \\
& \cong \displaystyle{\bigoplus_{r = 0}^k\bigoplus_{\lambda\vdash r}}\,  P_\lambda(n) \M_k^\lambda & \hbox{ as a $\C_k$-module}, \\
& \cong \displaystyle{\bigoplus_{r = 0}^k\bigoplus_{\lambda\vdash r}}\left(\V^\lambda \otimes \M_k^\lambda \right) & \hbox{ as an $(\O_n(\CC),\C_k)$-bimodule}. \\
\end{array}
\end{equation}
\noindent Note that it follows from these expressions that 
\begin{equation} \label{eq:bimoduleidentity}
(n+1)^k = \sum_{r = 0}^k \sum_{\lambda\vdash r} P_\lambda(n) m_k^\lambda.
\end{equation}

\item By general Wedderburn theory, the dimension of $\C_k$ is the sum of the squares of the dimensions of its irreducible modules,  
\begin{equation}\label{eq:evid}
\dim(\C_k) = \sum_{r = 0}^k \sum_{\lambda\vdash r} (m_k^\lambda)^2.
\end{equation}
\end{itemize}

\end{subsection}

\begin{subsection}{Action of $\RB_k(n+1)$ on tensor space}

Let $\KK=\CC$.
For $d$ in the set $\RBc_{k}$ of  rook-Brauer $k$-diagrams, we define an action of $d$ on the basis of simple tensors
in $\VV^{\otimes k}$  by
\begin{equation}\label{ActionOnTensorSpace}
d (v_{i_1} \otimes \cdots \otimes v_{i_k}) = \sum_{j_1, \ldots, j_k}  (d)_{i_1, \ldots, i_k}^{j_1, \ldots, j_k}\  v_{j_1} \otimes \cdots \otimes v_{j_k},
\end{equation}
where $(d)_{i_1, \ldots, i_k}^{j_1, \ldots, j_k}$ is computed by labeling the vertices in the bottom row of $d$ with ${i_1}, \ldots, {i_k}$ and the vertices in the top row of $d$ with ${j_1}, \ldots, {j_k}$.  Then
$$
(d)_{i_1, \ldots, i_k}^{j_1, \ldots, j_k} = \prod_{\varepsilon \in d} (\varepsilon)_{i_1, \ldots, i_k}^{j_1, \ldots, j_k},
$$
such that  the product is over the weights of all connected components $\varepsilon$ (edges and isolated vertices) in the diagram $d$, where by the weight of $\varepsilon$ we mean  
$$
(\varepsilon)_{i_1, \ldots, i_k}^{j_1, \ldots, j_k} = \begin{cases}
\delta_{a,0}, & \text{ if $\varepsilon$ is an isolated vertex labeled by $a$,} \\
\delta_{a,b}, &  \text{ if $\varepsilon$ is an edge in $d$ connecting $a$ and $b$,} \\
\end{cases}
$$
and $\delta_{a,b}$ is the Kronecker delta.
 For  example, for this labeled diagram
 $$
d = \begin{array}{c}
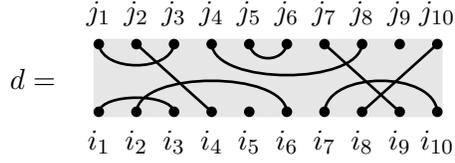
\begin{tikzpicture}[scale=.5,line width=1pt] 
\foreach \i in {1,...,10} 
{ \path (\i,.9) coordinate (T\i); \path (\i,-.9) coordinate (B\i); } 
%%%%%%%
\filldraw[fill= black!10,draw=black!10,line width=4pt]  (T1) -- (T10) -- (B10) -- (B1) -- (T1);
\draw[black] (T7) -- (B9);
\draw[black] (T10) -- (B8);
\draw[black] (T2) -- (B4);
\draw[black] (T1) .. controls +(.1,-.75) and +(-.1,-.75) .. (T3) ;
\draw[black] (T4) .. controls +(.1,-1.1) and +(-.1,-1.1) .. (T8) ;
\draw[black] (T5) .. controls +(.1,-.5) and +(-.1,-.5) .. (T6) ;
\draw[black] (B1) .. controls +(.1,.5) and +(-.1,.5) .. (B3) ;
\draw[black] (B7) .. controls +(.1,1.1) and +(-.1,1.1) .. (B10) ;
\draw[black] (B2) .. controls +(.1,1.1) and +(-.1,1.1) .. (B6) ;
%%%%%%%
\foreach \i in {1,...,10} 
{ \fill (T\i) circle (4pt); \fill (B\i) circle (4pt); } 
\draw  (T1)  node[above=0.1cm]{$j_1$};\draw  (B1)  node[below=0.1cm]{$i_1$};
\draw  (T2)  node[above=0.1cm]{$j_2$};\draw  (B2)  node[below=0.1cm]{$i_2$};
\draw  (T3)  node[above=0.1cm]{$j_3$};\draw  (B3)  node[below=0.1cm]{$i_3$};
\draw  (T4)  node[above=0.1cm]{$j_4$};\draw  (B4)  node[below=0.1cm]{$i_4$};
\draw  (T5)  node[above=0.1cm]{$j_5$};\draw  (B5)  node[below=0.1cm]{$i_5$};
\draw  (T6)  node[above=0.1cm]{$j_6$};\draw  (B6)  node[below=0.1cm]{$i_6$};
\draw  (T7)  node[above=0.1cm]{$j_7$};\draw  (B7)  node[below=0.1cm]{$i_7$};
\draw  (T8)  node[above=0.1cm]{$j_8$};\draw  (B8)  node[below=0.1cm]{$i_8$};
\draw  (T9)  node[above=0.1cm]{$j_9$};\draw  (B9)  node[below=0.1cm]{$i_9$};
\draw  (T10)  node[above=0.1cm]{$j_{10}$};\draw  (B10)  node[below=0.1cm]{$i_{10}$};
\end{tikzpicture}\end{array}
$$
we have 
$$
d_{i_1, \ldots, i_k}^{j_1, \ldots, j_k}= \delta_{j_1,j_3} \delta_{j_2,i_4}  \delta_{j_4,j_8}  \delta_{j_5,j_6}
 \delta_{j_7,i_9}  \delta_{j_9,0} \delta_{j_{10},i_8} 
 \delta_{i_1,i_3}  \delta_{i_2,i_6} \delta_{i_5,0} \delta_{i_7,i_{10}}.
$$
The representing transformation of $d$ is obtained by extending this action linearly to $\VV^{\otimes k}$.

Define $s, t: \V\otimes \V \to \V \otimes \V$ and $p: \V\to \V$  on a simple tensors  by
 \begin{equation}\label{eq:stpdef}
 \begin{array}{rcl}
s \cdot (v_a \otimes v_b) &=& v_b \otimes v_a,  \\
t \cdot (v_a \otimes v_b) &=& \displaystyle{\delta_{a,b} \sum_{k = 0}^n v_k \otimes v_k,} \\
p\cdot v_a &=& \delta_{a,0}\, v_0,
\end{array}
\end{equation}
Then the representing action of $s_i, t_i,$ and $p_i$ under \eqref{ActionOnTensorSpace} is given by  
 \begin{equation}\label{eq:Ract}
\begin{array}{lllclll}
s_i &\mapsto& \id_\VV^{\otimes i-1} \otimes s \otimes \id_\VV^{\otimes k - (i+1)}, \qquad &1 \le i < k, \\ \\
t_i  &\mapsto& \id_\VV^{\otimes i-1} \otimes t \otimes \id_\VV^{\otimes k - (i+1)},\qquad &1 \le i < k, \\ \\
p_i &\mapsto& \id_\VV^{\otimes i-1} \otimes p \otimes \id_\VV^{\otimes k - i},\qquad &1 \le i \le k,
\end{array}
\end{equation}
respectively, where $\id_\V: \V \to \V$ is the identity map. The action of the symmetric group generators $s_i$ 
 by tensor place permutation is the same as in classical Schur-Weyl duality. The action of $t_i$ is the same
 as defined by R.\ Brauer \cite{Br} for the Brauer algebra, and the action of the $p_i$ is the same as defined
 by Solomon \cite{So} for the rook-monoid.

\begin{prop}\label{IsARep} The map $\pi_k: \RB_k(n+1) \rightarrow \End(\VV^{\otimes k})$  afforded by the action  \eqref{ActionOnTensorSpace} of the basis diagrams is an algebra representation.
\end{prop}

\begin{proof}  It is possible to verify that the action of the generators \eqref{eq:Ract} satisfies the defining relations \eqref{Snrels}-\eqref{RBnrels}. This is done in \cite{MM}. To get the action on a general basis diagram \eqref{ActionOnTensorSpace}, one would then have to show that the action of the generators extends to \eqref{ActionOnTensorSpace}.  Our approach is to use the diagram calculus to show that the action on diagrams satisfies diagram multiplication. That is,  for diagrams $d_1, d_2 \in \RBc_{k}$ we show that
$$
(d_1 d_2)_{i_1, \ldots, i_k}^{j_1, \ldots, j_k} = \sum_{\ell_1, \ldots, \ell_k} 
(d_1)^{j_1, \ldots, j_k}_{\ell_1, \ldots, \ell_k}
(d_2)^{\ell_1, \ldots, \ell_k}_{i_1, \ldots,i_k}.
$$
We do so by  considering the edges of $d_1d_2$ case by case.

%%%%Isolated Vertex%%%%
\medskip\noindent
\textbf{Case 1:} \emph{Isolated vertices}. 
If there is an isolated vertex in column $r$ of the top row of $d_1$, then both $(d_1 d_2)_{i_1, \ldots, i_k}^{j_1, \ldots, j_k}$ and $(d_1)^{j_1, \ldots, j_k}_{\ell_1, \ldots, \ell_k}(d_2)^{\ell_1, \ldots, \ell_k}_{i_1, \ldots,i_k}$ contain $\delta_{i_r,0}$ for all choices of $\ell_1, \ldots, \ell_k$. On the other hand,  an isolated vertex in $d_1 d_2$ occurs when a vertical edge in $d_1$ is connected to a series of  horizontal edges (possibly 0 of them) in the middle row of $d_1d_2$ which end at an isolated vertex in the middle row of $d_1d_2$.  For example,
$$
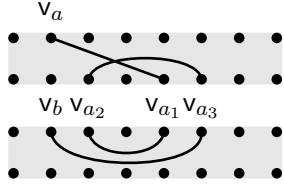
\begin{tikzpicture}[scale=.5,line width=1pt] 
\foreach \i in {1,...,8} 
{ \path (\i,1.8) coordinate (T\i); \path (\i,.7) coordinate (B\i); \path (\i,-.7) coordinate (T2\i); \path (\i,-1.8) coordinate (B2\i);} 
%%%%%%%
\filldraw[fill= black!10,draw=black!10,line width=4pt]  (T1) -- (T8) -- (B8) -- (B1) -- (T1);
\filldraw[fill= black!10,draw=black!10,line width=4pt]  (T21) -- (T28) -- (B28) -- (B21) -- (T21);
\draw[black] (T2) -- (B5);
\draw[black] (T23) .. controls +(.1,-.75) and +(-.1,-.75) .. (T25) ;
\draw[black] (B3) .. controls +(.1,.75) and +(-.1,.75) .. (B6) ;
\draw[black] (T26) .. controls +(.1,-1.1) and +(.1,-1.1) .. (T22) ;
%%%%%%%
\foreach \i in {1,...,8} 
{ \fill (T\i) circle (4pt); \fill (B\i) circle (4pt); \fill (T2\i) circle (4pt); \fill (B2\i) circle (4pt);} 
\draw  (T2)  node[above=0.1cm]{$\v_a$};\draw  (B5)  node[below=0.1cm]{$\v_{a_1}$};
\draw  (B3)  node[below=0.1cm]{$\v_{a_2}$};\draw  (B6)  node[below=0.1cm]{$\v_{a_3}$};
\draw  (B2)  node[below=0.1cm]{$\v_b$};
\end{tikzpicture}
$$
Sequentially label these vertices $\v_{a}, \v_{a_1}, \v_{a_3}, \dots, \v_{a_t}, \v_b$ as shown in the diagram above.  Then  $(d_1 d_2)_{i_1, \ldots, i_k}^{j_1, \ldots, j_k}$ contains $\delta_{a,0}$ and $(d_1)^{j_1, \ldots, j_k}_{\ell_1, \ldots, \ell_k}(d_2)^{\ell_1, \ldots, \ell_k}_{i_1, \ldots,i_k}$ contains $\delta_{a,a_1} \delta_{a_1,a_2} \cdots \delta_{a_t, b} \delta_{b, 0}$ which equals $\delta_{a,0}$ when $a_1  = \cdots = a_t =  b = 0$.
The proof for when the isolated edge is in the bottom row of $d_1d_2$ is analogous.

%%%% Edges%%%%%
\medskip\noindent
\textbf{Case 2: } \emph{Horizontal and vertical edges}.
A vertical edge in $d_1d_2$ occurs when a vertical edge in $d_2$ is connected to a vertical edge in $d_1$ by a series  of horizontal edges (possibly 0) in the middle row of $d_1d_2$. 
A horizontal edge in top row of $d_1d_2$ results from a horizontal edge in $t_1$ or two vertical edges in $d_1$ connected by a series of horizontal edges in the middle row of $d_1d_2$.  
For example,
$$
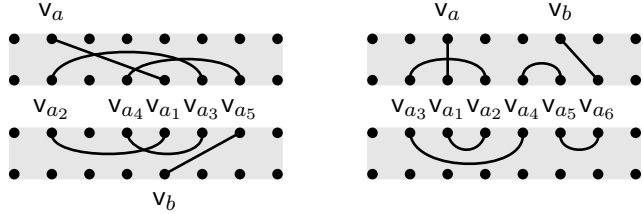
\begin{array}{c}
\begin{tikzpicture}[scale=.5,line width=1pt] 
\foreach \i in {1,...,8} 
{ \path (\i,1.8) coordinate (T\i); \path (\i,.7) coordinate (B\i); \path (\i,-.7) coordinate (T2\i); \path (\i,-1.8) coordinate (B2\i);} 
%%%%%%%
\filldraw[fill= black!10,draw=black!10,line width=4pt]  (T1) -- (T8) -- (B8) -- (B1) -- (T1);
\filldraw[fill= black!10,draw=black!10,line width=4pt]  (T21) -- (T28) -- (B28) -- (B21) -- (T21);
\draw[black] (T2) -- (B5);
\draw[black] (T25) .. controls +(.1,-.75) and +(-.1,-.75) .. (T22) ;
\draw[black] (B2) .. controls +(.1,1) and +(-.1,1) .. (B6) ;
\draw[black] (T26) .. controls +(.1,-.75) and +(.1,-.75) .. (T24) ;
\draw[black] (B4) .. controls +(.1,.75) and +(.1,.75) .. (B7);
\draw[black] (T27) -- (B25);
%%%%%%%
\foreach \i in {1,...,8} 
{ \fill (T\i) circle (4pt); \fill (B\i) circle (4pt); \fill (T2\i) circle (4pt); \fill (B2\i) circle (4pt);} 
\draw  (T2)  node[above=0.1cm]{$\v_a$};\draw  (B5)  node[below=0.1cm]{$\v_{a_1}$};
\draw  (B2)  node[below=0.1cm]{$\v_{a_2}$};\draw  (B6)  node[below=0.1cm]{$\v_{a_3}$};
\draw  (B4)  node[below=0.1cm]{$\v_{a_4}$};\draw  (B7)  node[below=0.1cm]{$\v_{a_5}$};
\draw  (B25)  node[below=0.1cm]{$\v_{b}$};
\end{tikzpicture}
\end{array}
\qquad
\begin{array}{c}
\begin{tikzpicture}[scale=.5,line width=1pt] 
\foreach \i in {0,...,7} 
{ \path (\i,1.8) coordinate (T\i); \path (\i,.7) coordinate (B\i); \path (\i,-.7) coordinate (T2\i); \path (\i,-1.8) coordinate (B2\i);} 
%%%%%%%
\filldraw[fill= black!10,draw=black!10,line width=4pt]  (T0) -- (T7) -- (B7) -- (B0) -- (T0);
\filldraw[fill= black!10,draw=black!10,line width=4pt]  (T20) -- (T27) -- (B27) -- (B20) -- (T20);
\draw[black] (T2) -- (B2);
\draw[black] (T22) .. controls +(.1,-.6) and +(-.1,-.6) .. (T23) ;
\draw[black] (B3) .. controls +(.1,.75) and +(-.1,.75) .. (B1) ;
\draw[black] (T21) .. controls +(.1,-1.1) and +(.1,-1.1) .. (T24) ;
\draw[black] (B4) .. controls +(-.1,.6) and +(.1,.6) .. (B5);
\draw[black] (T25) .. controls +(-.1,-.6) and +(.1, -.6) .. (T26);
\draw[black] (B6) -- (T5);
%%%%%%%
\foreach \i in {0,...,7} 
{ \fill (T\i) circle (4pt); \fill (B\i) circle (4pt); \fill (T2\i) circle (4pt); \fill (B2\i) circle (4pt);} 
\draw  (T2)  node[above=0.1cm]{$\v_a$};\draw  (B2)  node[below=0.1cm]{$\v_{a_1}$};
\draw  (B3)  node[below=0.1cm]{$\v_{a_2}$};\draw  (B1)  node[below=0.1cm]{$\v_{a_3}$};
\draw  (B4)  node[below=0.1cm]{$\v_{a_4}$};\draw  (B5)  node[below=0.1cm]{$\v_{a_5}$};
\draw  (B6)  node[below=0.1cm]{$\v_{a_6}$}; \draw (T5) node[above=0.1 cm]{$\v_{b}$};
\draw  (B25)  node[below=0.1cm]{$\vphantom{\v_{b}}$};
\end{tikzpicture}
\end{array}
$$
Label the top vertex in $d_1$ with $\v_a$ and sequentially label the connected vertices with $\v_{a_1}, \v_{a_2}, \dots , \v_{a_t}, \v_b$ as shown in the above diagrams.  Then  $(d_1 d_2)_{i_1, \ldots, i_k}^{j_1, \ldots, j_k}$ contains $\delta_{a,b}$ and $(d_1)^{j_1, \ldots, j_k}_{\ell_1, \ldots, \ell_k}(d_2)^{\ell_1, \ldots, \ell_k}_{i_1, \ldots,i_k}$ contains $\delta_{a,a_1} \delta_{a_1,a_2} \cdots \delta_{a_t, b}$ which equals $\delta_{a,b}$ when $a_1 = a_2 = \cdots = a_t = b$.
The proof for when the horizontal edge is in the bottom row is analogous.

%%%%%%%%Loop in the Middle%%%%%%%
\smallskip\noindent
\textbf{Case 3:} \emph{Loops}.
A loop in the middle row of $d_1d_2$ results from a series of  horizontal edges in the middle of $d_1d_2$.  For example,
$$
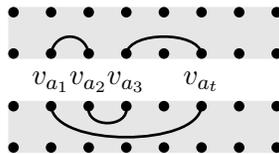
\begin{tikzpicture}[scale=.5,line width=1pt] 
\foreach \i in {0,...,7} 
{ \path (\i,1.8) coordinate (T\i); \path (\i,.7) coordinate (B\i); \path (\i,-.7) coordinate (T2\i); \path (\i,-1.8) coordinate (B2\i);} 
%%%%%%%
\filldraw[fill= black!10,draw=black!10,line width=4pt]  (T0) -- (T7) -- (B7) -- (B0) -- (T0);
\filldraw[fill= black!10,draw=black!10,line width=4pt]  (T20) -- (T27) -- (B27) -- (B20) -- (T20);
\draw[black] (B1) .. controls +(.1,.6) and +(-.1,.6) .. (B2) ;
\draw[black] (T22) .. controls +(-.1,-.6) and +(.1,-.6) .. (T23) ;
\draw[black] (B3) .. controls +(.1,.6) and +(-.1,.6) .. (B5) ;
\draw[black] (T25) .. controls +(-.1,-1.1) and +(.1,-1.1) .. (T21);
%%%%%%%
\foreach \i in {0,...,7} 
{ \fill (T\i) circle (4pt); \fill (B\i) circle (4pt); \fill (T2\i) circle (4pt); \fill (B2\i) circle (4pt);} 
\draw (B1) node[below=0.1 cm]{$v_{a_1}$}; \draw (B2) node[below=0.1 cm]{$v_{a_2}$};
 \draw (B3) node[below=0.1 cm]{$v_{a_3}$}; \draw (B5) node[below=0.1 cm]{$v_{a_t}$}; 
\end{tikzpicture}
$$
Starting with the left most middle vertex, label the vertices $v_{a_1}, v_{a_2}, \dots, v_{a_t}$. When multiplying $d_1d_2$ in $\RB_k(n)$ the inner loop is removed and the product is scaled by $n+1$ (see \eqref{multiplication}).  On the other hand, the coefficient $(d_1)^{j_1, \ldots, j_k}_{\ell_1, \ldots, \ell_k}(d_2)^{\ell_1, \ldots, \ell_k}_{i_1, \ldots,i_k}$ contains $\delta_{a_1,a_2} \delta_{a_2,a_3} \cdots \delta_{a_t,a_1}$ which is equals 1 when $a_1 = a_2 = \cdots = a_t$ is any of the $n+1$ values in $\{0,1, \ldots, n\}$ and is 0 otherwise. 
\end{proof}

\begin{prop} The representation $\pi_k: \RB_k(n+1) \to \End(\VV^{\otimes k})$ is faithful for $n \ge k$.
\end{prop}
\begin{proof} Suppose there exists some nonzero $y = \sum_{d\in \RBc_k} a_d\, d \in \Ker(\pi_k)$.  Choose a diagram $d'$ in the linear combination for $y$ such that

\begin{compactenum}
\item[(i)] $a_{d'}\neq 0$,
\item[(ii)] among the diagrams satisfying (i), $d'$ has a maximum number of vertical edges $m$, and
\item[(iii)] among the diagrams satisfying (i) and (ii), $d'$ has a maximum number of horizontal edges $\ell$.
\end{compactenum}

Now, consider the simple tensor $u$ such that 
(i) $\v_0$ is in the positions of the isolated vertices in the bottom row of $d'$,
(ii) $\v_1, \v_2, \dots, \v_m$ are in the positions of the bottom vertices of the vertical edges in $d'$, and
(iii) $\v_{m+1}, \v_{m+2}, \dots, \v_{m+\ell}$ are in the positions of the vertices of the horizontal edges in the bottom row of $d'$ such that the subscripts of the vectors in the positions of either end of a horizontal edge are the same.
Similarly, consider the simple tensor $u'$ such that 
(i) $\v_0$ is in the positions of the isolated vertices of the top row of $d'$,  
(ii) $\v_1, \v_2, \dots, \v_m$ are in the positions of the top vertices of the vertical edges in $d'$, and
(iii) $\v_{m+1}, \v_{m+2}, \dots, \v_{m+s}$ are in the positions of the vertices of the horizontal edges in the top row of $d'$ such that the subscripts of the vectors in the positions of either end of a horizontal edge are the same.

For example
$$
 
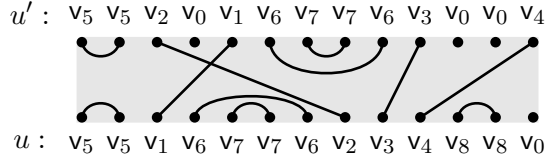
\begin{tikzpicture}[scale=.5,line width=1pt] 
\foreach \i in {0,...,13} 
{ \path (\i,1) coordinate (T\i); \path (\i,-1) coordinate (B\i); } 
%%%%%%%
\filldraw[fill= black!10,draw=black!10,line width=4pt]  (T1) -- (T13) -- (B13) -- (B1) -- (T1);
\draw[black] (T3) -- (B8);
\draw[black] (T10) -- (B9);
\draw[black] (T13) -- (B10);
\draw[black] (T5) -- (B3);
\draw[black] (T1) .. controls +(.1,-.5) and +(-.1,-.5) .. (T2) ;
\draw[black] (T6) .. controls +(.1,-1.1) and +(-.1,-1.1) .. (T9) ;
\draw[black] (T7) .. controls +(.1,-.5) and +(-.1,-.5) .. (T8) ;
\draw[black] (B1) .. controls +(.1,.5) and +(-.1,.5) .. (B2) ;
\draw[black] (B4) .. controls +(.1,.75) and +(-.1,.75) .. (B7) ;
\draw[black] (B5) .. controls +(.1,.5) and +(-.1,.5) .. (B6) ;
\draw[black] (B11) .. controls +(.1,.5) and +(-.1, .5) .. (B12);
%%%%%%%
\foreach \i in {1,...,13} 
{ \fill (T\i) circle (4pt); \fill (B\i) circle (4pt); } 
\draw  (T1)  node[above=0.1cm]{$\v_5$};\draw  (T2)  node[above=0.1cm]{$\v_5$};
\draw  (T3)  node[above=0.1cm]{$\v_2$};\draw  (T4)  node[above=0.1cm]{$\v_0$};
\draw  (T5)  node[above=0.1cm]{$\v_1$};\draw  (T6)  node[above=0.1cm]{$\v_6$};
\draw  (T7)  node[above=0.1cm]{$\v_7$};\draw  (T8)  node[above=0.1cm]{$\v_7$};
\draw  (T9)  node[above=0.1cm]{$\v_6$};\draw  (T10)  node[above=0.1cm]{$\v_3$};
\draw  (T11)  node[above=0.1cm]{$\v_0$};\draw  (T12)  node[above=0.1cm]{$\v_0$};
\draw  (T13)  node[above=0.1cm]{$\v_4$};\draw  (B1)  node[below=0.1cm]{$\v_5$};
\draw  (B2)  node[below=0.1cm]{$\v_5$};\draw  (B3)  node[below=0.1cm]{$\v_1$};
\draw  (B4)  node[below=0.1cm]{$\v_6$};\draw  (B5)  node[below=0.1cm]{$\v_7$};
\draw  (B6)  node[below=0.1cm]{$\v_7$};\draw  (B7)  node[below=0.1cm]{$\v_6$};
\draw  (B8)  node[below=0.1cm]{$\v_2$};\draw  (B9)  node[below=0.1cm]{$\v_3$};
\draw  (B10)  node[below=0.1cm]{$\v_4$};\draw  (B11)  node[below=0.1cm]{$\v_8$};
\draw  (B12)  node[below=0.1cm]{$\v_8$};\draw  (B13)  node[below=0.1cm]{$\v_0$};
\draw (B0)  node[below=0.1cm]{$u:\quad $};\draw  (T0)  node[above=0.1cm]{$u':\quad $};
\end{tikzpicture}
$$
Note that the hypothesis $n\geq k$ guarantees that such simple tensors $u$ and $u'$ exist.

By this construction,  the simple tensor $u'$ has coefficient 1 in the expansion of $d' \cdot u$.  We claim that no other diagram with nonzero coefficient in $y$ will act on $u$ with a nonzero coefficient of $u'$.  In order for another diagram $d''$ in $y$ to possibly produce a nonzero coefficient of $u'$ when acting on $u$, $d''$ must have the same bottom row as $d'$.  This follows from the choice of $d'$ having the maximum number of vertical edges and horizontal edges and choosing distinct $\v_i$ to put in each position of the edges of $d'$.  The same conditions force the top row of $d'$ and $d''$ to be the same.  Therefore, $d'' = d'$ and thus  $y\cdot u \not= 0$, which contradicts the assumption that $y\in \Ker(\pi_k)$.  Thus $\Ker(\pi_k) = 0$ and $\pi_k$ is faithful. \end{proof}

\begin{thm}\label{thm:commutingact}  Let  $\rho_{\VV^{\otimes k}}: \O_n(\CC) \rightarrow  \End(\VV^{\otimes k})$ 
denote the representation of $\O_n(\CC)$  on $\VV^{\otimes k}.$   Then  $\pi_k( \RB_k(n+1))$ and   $\rho_{\VV^{\otimes k}} (\O_n(\CC))$ commute in $\End(\VV^{\otimes k})$.
Thus, $\pi_k( \RB_k(n+1))\subseteq \End_{\O_n(\CC)}(\VV^{\otimes k})$ and  
$\rho_{\VV^{\otimes k}}(\O_n(\CC)) \subseteq \End_{ \RB_k(n+1)}(\VV^{\otimes k}).$   
\end{thm}

\begin{proof} The elements $t_i, s_i, p_j (1 \le i <k, 1 \le j \le k)$ generate $\RB_k(n)$.  Let $g \in \O_n(\CC)$ have matrix entries $g_{ij}$ so that $g \cdot \v_j = \sum_{i = 1}^n g_{ij} \, \v_i$ and $g \cdot \v_0 = \v_0$.  The fact that $gg^T = I_n$ since $g \in \O_n(\CC)$ tells us that $\sum_{\ell = 1}^n  g_{i\ell} g_{j\ell} = \delta_{ij}$. From \eqref{eq:stpdef}  and \eqref{eq:Ract} we see that  $s_i, t_i$ act only on tensor positions $i$ and $i+1$ and $p_j$ acts only on tensor position $j$ (i.e., these generators act as the identity elsewhere), so it is sufficient to verify that $g$ commutes with $s$ and $t$  on $\V \otimes \V$ and that $g$ commutes with $p$ on $\V$.  These are straight-forward calculations and we provide one illustrative calculation here.  Let $1 \le a,b \le n$. Then
\begin{align*}
g t \cdot(\v_a \otimes \v_b) &= \delta_{ab} \sum_{\ell = 0}^n g (\v_\ell \otimes \v_\ell) 
= \delta_{ab}(\v_0 \otimes \v_0) +  \delta_{ab}\sum_{\ell = 1}^n  \sum_{i=1}^n \sum_{j=1}^n g_{i\ell} g_{j\ell} (\v_i \otimes \v_j) \\
& =  \delta_{ab}(\v_0 \otimes \v_0) +  \delta_{ab}\sum_{i = 1}^n  (\v_i \otimes \v_i) =  \delta_{ab}\sum_{i = 0}^n  (\v_i \otimes \v_i),
\\
tg  \cdot(\v_a \otimes \v_b) &= \sum_{i=1}^n \sum_{j = 1}^n g_{ia} g_{jb} t(\v_i \otimes \v_j) 
= \sum_{i=1}^n  g_{ia} g_{ib} t(\v_i \otimes \v_i)  \\
&= \delta_{ab} t(\v_i \otimes \v_i)=  \delta_{ab}\sum_{\ell = 0}^n  (\v_i \otimes \v_i),
\end{align*}
The cases where $a=0$ or $b=0$  are even easier to verify and they must be considered separately since $g$ acts differently on $\v_0$. \end{proof}

The previous theorem gives a faithful embedding of $\RB_k(n+1)$ in  $\End_{\O_n(\CC)}(\V^{\otimes k})$ for $n \ge  k$.  In the next section, we prove that the dimension of the centralizer algebra equals the dimension of $\RB_k(n+1)$ and therefore
\begin{equation}\label{SchurWeyl}
\RB_k(n+1) \cong \End_{\O_n(\CC)}(\V^{\otimes k}), \qquad  n \ge k.
\end{equation}

\end{subsection}

\begin{subsection}{RSK insertion, paths, and vacillating tableaux}\label{sec:vacillating}

For $\lambda \in \Lambda_k$ we define a \emph{rook-Brauer path} (or simply a \emph{path}) of shape $\lambda$ and length $k$ to be a path of length $k$  in the Bratteli diagram $\mathcal B$ from $\emptyset \in \Lambda_0$ to $\lambda \in \Lambda_k$.    These paths exactly correspond to  sequences
$( \lambda^0, \lambda^1,  \ldots, \lambda^k)$ of integer partitions such that 
(i) $\lambda^0 = \emptyset$,
(ii)  $\lambda^k = \lambda$, and 
(iii) $\lambda^{i+1}$ is obtained from $\lambda^i$ by adding a box ($\lambda^{i+1} = \lambda^i + \square$), removing a box ($\lambda^{i+1} = \lambda^i - \square$), or doing nothing ($\lambda^{i+1} = \lambda^i$).   As an example, the following is a path of shape $\lambda = (2)$ and length $k = 14$,
$$
\left( \emptyset, 
\begin{array}{c}\begin{tikzpicture}[scale=.25]  \draw (0,0) -- (1,0) -- (1,1) -- (0,1) -- (0,0); \end{tikzpicture}\end{array},
\begin{array}{c}\begin{tikzpicture}[scale=.25]  \draw (0,0) -- (1,0) -- (1,1) -- (0,1) -- (0,0); \end{tikzpicture}\end{array},
\emptyset,
\begin{array}{c}\begin{tikzpicture}[scale=.25]  \draw (0,0) -- (1,0) -- (1,1) -- (0,1) -- (0,0); \end{tikzpicture}\end{array},
\begin{array}{c}\begin{tikzpicture}[scale=.25] 
\draw (0,0) -- (1,0) -- (1,1) -- (0,1) -- (0,0);  \draw (1,0) -- (2,0) -- (2,1) -- (1,1);
 \end{tikzpicture}\end{array},
 \begin{array}{c}\begin{tikzpicture}[scale=.25] 
\draw (0,0) -- (1,0) -- (1,1) -- (0,1) -- (0,0);  \draw (1,0) -- (2,0) -- (2,1) -- (1,1);  \draw (0,0) -- (0,-1) -- (1,-1) -- (1,0);
 \end{tikzpicture}\end{array},
  \begin{array}{c}\begin{tikzpicture}[scale=.25] 
\draw (0,0) -- (1,0) -- (1,1) -- (0,1) -- (0,0);   \draw (0,0) -- (0,-1) -- (1,-1) -- (1,0);
 \end{tikzpicture}\end{array},
   \begin{array}{c}\begin{tikzpicture}[scale=.25] 
\draw (0,0) -- (1,0) -- (1,1) -- (0,1) -- (0,0);   \draw (0,0) -- (0,-1) -- (1,-1) -- (1,0);
 \end{tikzpicture}\end{array},
  \begin{array}{c}\begin{tikzpicture}[scale=.25] 
\draw (0,0) -- (1,0) -- (1,1) -- (0,1) -- (0,0);  \draw (1,0) -- (2,0) -- (2,1) -- (1,1);  \draw (0,0) -- (0,-1) -- (1,-1) -- (1,0);
 \end{tikzpicture}\end{array},
   \begin{array}{c}
\begin{tikzpicture}[scale=.25] 
\draw (0,0) -- (1,0) -- (1,1) -- (0,1) -- (0,0);  \draw (1,0) -- (2,0) -- (2,1) -- (1,1);  \draw (0,0) -- (0,-1) -- (1,-1) -- (1,0); \draw (1,-1) -- (2,-1) -- (2,0) -- (1,0); 
 \end{tikzpicture}\end{array},
   \begin{array}{c}\begin{tikzpicture}[scale=.25] 
\draw (0,0) -- (1,0) -- (1,1) -- (0,1) -- (0,0);  \draw (1,0) -- (2,0) -- (2,1) -- (1,1);  \draw (0,0) -- (0,-1) -- (1,-1) -- (1,0);
 \end{tikzpicture}\end{array},
   \begin{array}{c}\begin{tikzpicture}[scale=.25] 
\draw (0,0) -- (1,0) -- (1,1) -- (0,1) -- (0,0);  \draw (1,0) -- (2,0) -- (2,1) -- (1,1);   \draw (2,0) -- (3,0) -- (3,1) -- (2,1);  \draw (0,0) -- (0,-1) -- (1,-1) -- (1,0);
 \end{tikzpicture}\end{array},
   \begin{array}{c}\begin{tikzpicture}[scale=.25] 
\draw (0,0) -- (1,0) -- (1,1) -- (0,1) -- (0,0);  \draw (1,0) -- (2,0) -- (2,1) -- (1,1);   \draw (2,0) -- (3,0) -- (3,1) -- (2,1); 
 \end{tikzpicture}\end{array},
 \begin{array}{c}\begin{tikzpicture}[scale=.25] 
\draw (0,0) -- (1,0) -- (1,1) -- (0,1) -- (0,0);  \draw (1,0) -- (2,0) -- (2,1) -- (1,1);
 \end{tikzpicture}\end{array}
\right).
$$
These paths are very closely related to the \emph{vacillating tableaux} studied in \cite{HL} and   \cite{CDDSY}. 
Let $\mathcal{T}_k(\lambda)$ denote the set of  paths of shape $\lambda$ and length $k$.  
Then from \eqref{eq:DimensionPaths} we have
$$
|\mathcal{T}_k(\lambda)| = m_k^\lambda = \dim(\M_k^\lambda) = \mult_{\V^{\otimes k}} (\V(\lambda)).
$$

  The rook-Brauer diagrams in this paper are special cases of the set partition diagrams in the partition algebra defined in \cite{Ma}, \cite{Jo}.    In \cite{HL}, a Robinson-Schensted-Knuth insertion algorithm is given that turns set partition diagrams to pairs of paths in the Bratteli diagram of the partition algebra.   When the algorithm is restricted  to rook-Brauer diagrams, at each step we either add a box, remove a box, or stay the same, and thus we obtain a bijection between rook-Brauer diagrams and pairs of rook-Brauer paths $(S,T)$ of shape $\lambda \in \Lambda_k$ and length $k$.  We illustrate this bijection in Example \ref{RSK} below.

Our bijection gives us the following dimension formula,
\begin{equation}
\dim(\RB_k(x)) = |\RBc_k| = \sum_{r = 0}^k \sum_{\lambda \vdash r} |\mathcal{T}_k(\lambda)|^2 = \sum_{r = 0}^k \sum_{\lambda \vdash r} (m_k^\lambda)^2 = \dim( \End_{\O_n(\CC)}(\V^{\otimes k})),
\end{equation}
which completes the proof of \eqref{SchurWeyl}.

\begin{examp}\label{RSK}{\rm
In this example, we illustrate the bijection of \cite{HL} that turns rook-Brauer diagrams into pairs of rook-Brauer paths. First we label our diagram $1, \ldots, k$ in the top row and $2k, \ldots, k+1$ in the bottom row and unfold our diagram to put it in a single row:
$$
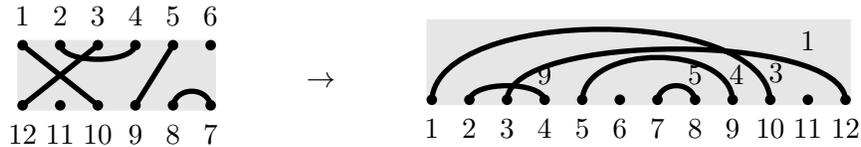
\begin{array}{c}
\begin{tikzpicture}[scale=.5,line width=2pt] 
\foreach \i in {1,...,6}
{ \path (\i,.8) coordinate (T\i); \path (\i,-.8) coordinate (B\i); } 
\filldraw[fill= black!10,draw=black!10,line width=4pt]  (T1) -- (T6) -- (B6) -- (B1) -- (T1);
%%%%%%%%%%%%%%%%%%%%%%%%%
\draw (T1) -- (B3);
\draw (T3) -- (B1);
\draw (T5) -- (B4);
\draw (T2) .. controls +(.1,-.5) and +(-.1,-.5) .. (T4);
\draw (B5) .. controls +(.1,.5) and +(-.1,.5) .. (B6);
%%%%%%%%%%%%%%%%%%%%%%%%%
\foreach \i in {1,...,6}
{\fill (T\i) circle (4pt); \fill (B\i) circle (4pt); \draw (T\i) node[above=0.1cm]{$\i$};}
\draw (B1) node[below=0.1cm]{12}; \draw (B2) node[below=0.1cm]{11}; \draw (B3) node[below = 0.1cm]{10};
\draw (B4) node[below=0.1cm]{9}; \draw (B5) node[below=0.1cm]{8}; \draw (B6) node[below=0.1cm]{7};
\end{tikzpicture}
\end{array}
\qquad\rightarrow\qquad
\begin{array}{c}
\begin{tikzpicture}[scale=.5,line width=2pt]
\foreach \i in {1,...,12}
{\path (\i, 0) coordinate (T\i);}
\filldraw[fill=black!10,draw=black!10,line width=4pt] (T1) -- (T12) -- (12,2) -- (1,2) -- (T1);
%%%%%%%%%%%%%%%%%
\draw (T1) .. controls +(.2,2.5) and +(-.2,2.5) .. (T10);
\draw (T2) .. controls +(.1,.5) and +(-.1,.5) .. (T4);
\draw (T3) .. controls +(.1,1.8) and +(-.1,1.8) .. (T12);
\draw (T5) .. controls +(.1,1.5) and +(-.1,1.5) .. (T9);
\draw (T7) .. controls +(.1,.5) and +(-.1,.5) .. (T8);
%%%%%%%%%%%%%%%%%%%
\foreach \i in {1,...,12}
{\fill (T\i) circle (4pt); \draw (T\i) node[below=0.1cm]{\i};}
\draw (11,1) node[above]{1};
\draw (10.15,.15) node[above]{3};
\draw (9.1,.1) node[above]{4};
\draw (8,.1) node[above]{5};
\draw (4, .1) node[above]{9};
\end{tikzpicture}
\end{array}
$$
We then label each edge in the diagram by $2 k - i + 1$, if $i$ is the right endpoint of the edge.

If vertex $i$ is the left endpoint of  edge $a$, we let $E_i=a_L$, if vertex $i$ is the right endpoint of  edge $a$, we let $E_i=a_R$, and if vertex $i$ is not incident to an edge, we let $E_i=\emptyset$. In this way we get an insertion sequence of the form,
$$(E_1, E_2, \ldots, E_{2k}) = ( 3_L, 9_L, 1_L, 9_R, 4_L, \emptyset, 5_L, 5_R, 4_R, 3_R, \emptyset, 1_R).$$
Now let $T^{(0)} = \emptyset$  and recursively define
$$
T^{(i)} = \begin{cases}
E_i \xrightarrow{\scriptstyle{RSK}} T^{(i-1)}, & \text{ if } E_i = a_L\\
E_i \xleftarrow{\,\,\,jdt\,\,} T^{(i-1)}, & \text { if } E_i = a_R\\
\end{cases}
$$
where $RSK$ denotes Robinson-Schensted-Knuth column insertion and $jdt$ denotes jeu de taquin (for details on these well-known combinatorial algorithms see \cite{HL} and the references therein).  In our running example, we insert and delete as follows:
$$
\begin{tabular}{rlcl c rlcl}
1)& $3_L$ & $\xrightarrow{RSK}$& $\begin{array}{l}\begin{tikzpicture}[scale=.35]  \draw (0,0) -- (1,0) -- (1,1) -- (0,1) -- (0,0);\draw (.5,.5) node{$\scriptstyle{3}$};   \end{tikzpicture}    \end{array}$                      
&\qquad\qquad\qquad& 
7)& $5_L$ & $\xrightarrow{RSK}$ &    $\begin{array}{l}\begin{tikzpicture}[scale=.35] 
\draw (0,0) -- (1,0) -- (1,1) -- (0,1) -- (0,0);\draw (.5,.5) node{$\scriptstyle{1}$}; 
 \draw (1,0) -- (2,0) -- (2,1) -- (1,1) -- (1,0); \draw (1.5,.5) node{$\scriptstyle{4}$};  
  \draw (2,0) -- (3,0) -- (3,1) -- (2,1) -- (2,0); \draw (2.5,.5) node{$\scriptstyle{5}$};  
\draw (0,-1) -- (1,-1) -- (1,0) -- (0,0) -- (0,-1);\draw (.5,-.5) node{$\scriptstyle{3}$};  
 \end{tikzpicture} \end{array}$ \\
2)& $9_L$ & $\xrightarrow{RSK}$&    
$\begin{array}{l}\begin{tikzpicture}[scale=.35] 
\draw (0,0) -- (1,0) -- (1,1) -- (0,1) -- (0,0);\draw (.5,.5) node{$\scriptstyle{3}$};  \draw (1,0) -- (2,0) -- (2,1) -- (1,1) -- (1,0); \draw (1.5,.5) node{$\scriptstyle{9}$};  \end{tikzpicture} \end{array}$               &&  
8)& $5_R$ & $\xleftarrow{\,\,\,jdt\,\,}$ &  
$\begin{array}{l}\begin{tikzpicture}[scale=.35] 
\draw (0,0) -- (1,0) -- (1,1) -- (0,1) -- (0,0);\draw (.5,.5) node{$\scriptstyle{1}$}; 
 \draw (1,0) -- (2,0) -- (2,1) -- (1,1) -- (1,0); \draw (1.5,.5) node{$\scriptstyle{4}$};  
\draw (0,-1) -- (1,-1) -- (1,0) -- (0,0) -- (0,-1);\draw (.5,-.5) node{$\scriptstyle{3}$};  
 \end{tikzpicture}\end{array}$ \\
3)& $1_L$ & $\xrightarrow{RSK}$&  $\begin{array}{l}\begin{tikzpicture}[scale=.35] 
\draw (0,0) -- (1,0) -- (1,1) -- (0,1) -- (0,0);\draw (.5,.5) node{$\scriptstyle{1}$}; 
 \draw (1,0) -- (2,0) -- (2,1) -- (1,1) -- (1,0); \draw (1.5,.5) node{$\scriptstyle{9}$};  
\draw (0,-1) -- (1,-1) -- (1,0) -- (0,0) -- (0,-1);\draw (.5,-.5) node{$\scriptstyle{3}$};  
 \end{tikzpicture}\end{array}$               &&  
9)& $4_R$ & $\xleftarrow{\,\,\,jdt\,\,}$ & $\begin{array}{l}\begin{tikzpicture}[scale=.35] 
\draw (0,0) -- (1,0) -- (1,1) -- (0,1) -- (0,0);\draw (.5,.5) node{$\scriptstyle{1}$}; 
\draw (0,-1) -- (1,-1) -- (1,0) -- (0,0) -- (0,-1);\draw (.5,-.5) node{$\scriptstyle{3}$};  
 \end{tikzpicture}\end{array}$\\
4)& $9_R$ & $\xleftarrow{\,\,\,jdt\,\,}$&   $\begin{array}{l}\begin{tikzpicture}[scale=.35] 
\draw (0,0) -- (1,0) -- (1,1) -- (0,1) -- (0,0);\draw (.5,.5) node{$\scriptstyle{1}$}; 
\draw (0,-1) -- (1,-1) -- (1,0) -- (0,0) -- (0,-1);\draw (.5,-.5) node{$\scriptstyle{3}$};  
 \end{tikzpicture}\end{array}$                 &&  
10)& $3_R$ & $\xleftarrow{\,\,\,jdt\,\,}$ & $\begin{array}{l}\begin{tikzpicture}[scale=.35] 
\draw (0,0) -- (1,0) -- (1,1) -- (0,1) -- (0,0);\draw (.5,.5) node{$\scriptstyle{1}$};  
 \end{tikzpicture}\end{array}$\\
5)& $4_L$ & $\xrightarrow{RSK}$& $\begin{array}{l}\begin{tikzpicture}[scale=.35] 
\draw (0,0) -- (1,0) -- (1,1) -- (0,1) -- (0,0);\draw (.5,.5) node{$\scriptstyle{1}$}; 
 \draw (1,0) -- (2,0) -- (2,1) -- (1,1) -- (1,0); \draw (1.5,.5) node{$\scriptstyle{4}$};  
\draw (0,-1) -- (1,-1) -- (1,0) -- (0,0) -- (0,-1);\draw (.5,-.5) node{$\scriptstyle{3}$};  
 \end{tikzpicture} \end{array}$              &&  
11)& $\emptyset$ & & $\begin{array}{l}\begin{tikzpicture}[scale=.35] 
\draw (0,0) -- (1,0) -- (1,1) -- (0,1) -- (0,0);\draw (.5,.5) node{$\scriptstyle{1}$}; 
 \end{tikzpicture}\end{array}$\\
6)& $\emptyset $ & &  $\begin{array}{l} \begin{tikzpicture}[scale=.35] 
\draw (0,0) -- (1,0) -- (1,1) -- (0,1) -- (0,0);\draw (.5,.5) node{$\scriptstyle{1}$}; 
 \draw (1,0) -- (2,0) -- (2,1) -- (1,1) -- (1,0); \draw (1.5,.5) node{$\scriptstyle{4}$};  
\draw (0,-1) -- (1,-1) -- (1,0) -- (0,0) -- (0,-1);\draw (.5,-.5) node{$\scriptstyle{3}$};  
 \end{tikzpicture}\end{array}$               &&  
12)& $1_R$ & $\xleftarrow{\,\,\,jdt\,\,}$ & $\emptyset$\\
\end{tabular}
$$
We then take the first $k$ shapes and the last $k$ shapes (reversed) to be our pair of paths
\begin{align*}
P &= (
\emptyset,
\begin{array}{l}\begin{tikzpicture}[scale=.25] 
\draw (0,0) -- (1,0) -- (1,1) -- (0,1) -- (0,0);
 \end{tikzpicture} \end{array},
 \begin{array}{l}\begin{tikzpicture}[scale=.25] 
\draw (0,0) -- (1,0) -- (1,1) -- (0,1) -- (0,0);\draw (1,0) -- (2,0) -- (2,1) -- (1,1) -- (1,0); 
 \end{tikzpicture} \end{array},
 \begin{array}{l}\begin{tikzpicture}[scale=.25] 
\draw (0,0) -- (1,0) -- (1,1) -- (0,1) -- (0,0); \draw (1,0) -- (2,0) -- (2,1) -- (1,1) -- (1,0);  
\draw (0,-1) -- (1,-1) -- (1,0) -- (0,0) -- (0,-1);
 \end{tikzpicture} \end{array},
  \begin{array}{l}\begin{tikzpicture}[scale=.25] 
\draw (0,0) -- (1,0) -- (1,1) -- (0,1) -- (0,0); 
\draw (0,-1) -- (1,-1) -- (1,0) -- (0,0) -- (0,-1);
 \end{tikzpicture} \end{array},
  \begin{array}{l}\begin{tikzpicture}[scale=.25] 
\draw (0,0) -- (1,0) -- (1,1) -- (0,1) -- (0,0); \draw (1,0) -- (2,0) -- (2,1) -- (1,1) -- (1,0);  
\draw (0,-1) -- (1,-1) -- (1,0) -- (0,0) -- (0,-1);
 \end{tikzpicture} \end{array},
  \begin{array}{l}\begin{tikzpicture}[scale=.25] 
\draw (0,0) -- (1,0) -- (1,1) -- (0,1) -- (0,0); \draw (1,0) -- (2,0) -- (2,1) -- (1,1) -- (1,0);  
\draw (0,-1) -- (1,-1) -- (1,0) -- (0,0) -- (0,-1);
 \end{tikzpicture} \end{array}), \\
Q &= (
\emptyset,
\begin{array}{l}\begin{tikzpicture}[scale=.25] 
\draw (0,0) -- (1,0) -- (1,1) -- (0,1) -- (0,0); 
 \end{tikzpicture} \end{array},
\begin{array}{l}\begin{tikzpicture}[scale=.25] 
\draw (0,0) -- (1,0) -- (1,1) -- (0,1) -- (0,0); 
 \end{tikzpicture} \end{array},
\begin{array}{l}\begin{tikzpicture}[scale=.25] 
\draw (0,0) -- (1,0) -- (1,1) -- (0,1) -- (0,0); 
\draw (0,-1) -- (1,-1) -- (1,0) -- (0,0) -- (0,-1);
 \end{tikzpicture} \end{array},
\begin{array}{l}\begin{tikzpicture}[scale=.25] 
\draw (0,0) -- (1,0) -- (1,1) -- (0,1) -- (0,0); \draw (1,0) -- (2,0) -- (2,1) -- (1,1) -- (1,0);  
\draw (0,-1) -- (1,-1) -- (1,0) -- (0,0) -- (0,-1);
 \end{tikzpicture} \end{array}
,
\begin{array}{l}\begin{tikzpicture}[scale=.25] 
\draw (0,0) -- (1,0) -- (1,1) -- (0,1) -- (0,0); \draw (1,0) -- (2,0) -- (2,1) -- (1,1) -- (1,0);   \draw (2,0) -- (3,0) -- (3,1) -- (2,1) -- (2,0); 
\draw (0,-1) -- (1,-1) -- (1,0) -- (0,0) -- (0,-1);
 \end{tikzpicture} \end{array},
\begin{array}{l}\begin{tikzpicture}[scale=.25] 
\draw (0,0) -- (1,0) -- (1,1) -- (0,1) -- (0,0); \draw (1,0) -- (2,0) -- (2,1) -- (1,1) -- (1,0);  
\draw (0,-1) -- (1,-1) -- (1,0) -- (0,0) -- (0,-1);
 \end{tikzpicture} \end{array}).
\end{align*}}
\end{examp}

\end{subsection}

\begin{subsection}{Semisimplicity of $\RB_k(x)$ for generic $x \in \CC$}\label{semisimplesection}

The coefficients of the relations  in the presentation of $\RB_k(x)$ are polynomials in $x$, and $\RB_k(x)$ is semisimple for each  $x \in \ZZ$ with $x \ge k$ (i.e., infinitely many). Thus, it follows from the Tits deformation theorem (see for example  \cite[Theorem 5.13]{HR} or \cite[68.7]{CR}) that
\begin{enumerate}
\item[(a)]  $\RB_k(x)$ is semisimple for all but a finite number of $x\in \CC$.
\item[(b)]  If $x \in \CC$ is such that $\RB_k(x)$ is semisimple, then
\begin{enumerate}
\item  $\Lambda_k$ is an index set for the irreducible $\RB_k(x)$-modules.
\item If  $\M^\lambda_k$ is the irreducible $\RB_k(x)$-module labeled by $\lambda \in \Lambda_k$, then $\dim(\V^\lambda) = m_k^\lambda.$
\end{enumerate}
\end{enumerate}
\end{subsection}

\end{section}

\begin{section}{Seminormal Representations}

Throughout this section let $\KK = \CC$ and let $x\in \CC$ be chosen such that $\RB_k(x)$ is semisimple.  We explicitly construct $\RB_k(x)$-modules $\M_k^\lambda$,  for each $\lambda \in \Lambda_k$. These representations have bases labeled by the paths to $\lambda$ in the Bratteli diagram and they are generalizations of the seminormal representations of the symmetric group (due to Young \cite{Yo}), of the Brauer algebra (due to Nazarov \cite{Nz} and Leduc and Ram \cite{LR}), and of the rook monoid algebra (due to Halverson \cite{Ha}). The matrices of the representations that factor through the symmetric group are orthogonal with respect to this basis.

We give explicit actions of our generators $p_i, t_i, s_i$ on the path basis,  and we show that if $x$ is chosen so that $\RB_k(x)$ is semisimple, then the modules $\M_k^\lambda$ form a complete set of pairwise nonisomorphic representations.   Furthermore, as shown in \eqref{restrictionrule} the seminormal bases constructed here as paths in the Bratteli diagram explicitly realize the restriction rules given by the Bratteli diagram. This is a defining feature of a seminormal basis as discussed in  \cite{Ra}.

\begin{subsection} {El-Samra King Polynomials}

For a partition $\lambda$, the dimension of the irreducible $\O_n(\CC)$ module $\V^\lambda$ is given by the El-Samra King polynomial \cite{El-K} (see also \cite[(6.21)]{LR}).  These polynomials occur in the matrix entries in the seminormal representations of $\RB_n(x)$ that we determine below.   For a partition $\lambda$ let $\lambda_i$ denote the length of the $i$th column and let $\lambda_j'$ denote the length of the $j$th column.  Define the hook of the box at position $(i,j) \in \lambda$ to be
$$
h(i,j) = \lambda_i - i + \lambda_j' - j + 1.
$$
and, for each box $(i,j) \in \lambda$ define
\begin{equation}
d(i,j) = \begin{cases}
 \lambda_i  + \lambda_j  - i -  j + 1, & \text{if $i \le j$}, \\
- \lambda_i'  - \lambda_j'  + i +  j - 1, & \text{if $i > j$}. \\
\end{cases}
\end{equation}
The El-Samra King polynomials  are given by
\begin{equation}\label{ElSamraKing}
P_\lambda(x) = \prod_{(i,j) \in \lambda} \frac{ x - 1 + d(i,j)}{h(i,j)},
\end{equation}
and the irreducible $\O_n(\CC)$-module $\V^\lambda$ has dimension given by  \cite{El-K}
$$
\dim(\V^\lambda) = P_\lambda(n).
$$

\end{subsection}

\begin{subsection} {Action of  generators on paths}

In this section we describe the action of each of our generators $t_i, s_i, p_i$ of $\RB_k(x)$ on a basis indexed by the paths $\mathcal{T}_k(\lambda)$ defined in  \ref{sec:vacillating}.   For each $\lambda \in \Lambda_k$, let $\{ \, \v_S \, \vert \, S \in \mathcal{T}_k(\lambda)\}$ be a set of linearly independent vectors indexed by $\mathcal{T}_k(\lambda)$, and let
$$
\M_k^\lambda = \CC\text{-span}\{\ \v_T\ \vert\ T \in \mathcal{T}_k(\lambda)\ \}
$$
be a vector space spanned by these vectors.

If $S = (\lambda^0, \lambda^1, \ldots, \lambda^k) \in \mathcal{T}_k(\lambda)$, then for each $1 \le j \le k$   define
\begin{equation}\label{p-action}
p_j \cdot \v_S = \begin{cases}
\v_S, & \text{ if } \lambda^{j-1} = \lambda^j,\\
0, & \text{ otherwise. }
\end{cases} 
\end{equation}
Observe that the action of $p_j$ depends only on the shapes in positions $i-1$ and $i$ of  $S$.

For $1 \le i < k$, the action of the generators $t_i$ and $s_i$ depend only on the partitions in positions $i-1,i,i+1$ of  a path $S\in \mathcal{T}_k(\lambda)$.   For this reason, we use the notation,
$$
\hbox{if} \quad  S = (\lambda^{0}, \ldots, \lambda^{k}) \quad\hbox{then}\quad S_i=(\lambda^{i-1},\lambda^i,\lambda^{i+1}), \quad 0<i<k.
$$
Furthermore, we say that a pair of paths $S = (\lambda^{0}, \ldots, \lambda^{k})$ and $T = (\tau^{0}, \ldots, \tau^{k})$ are \emph{$i$-compatible} if $\sigma^{j} = \tau^{j}$ whenever $j \not= i$.  That is, $S$ and $T$ can only differ in the $i$th position. We let $\Gamma_i(S)$ denote the set of all paths that are $i$-compatible with $S$. Note that $S \in \Gamma_i(S)$, and if $T \in \Gamma_i(S)$, then $S$ and $T$ have the form
\begin{align*}
S  &= \big(\, \lambda^{(0)}  \ldots, \lambda^{(i-2)},  \gamma, \alpha,\rho, \lambda^{(i+2)}, \ldots, \lambda^{(k)}\, \big), \quad\Rightarrow\quad S_i = \big( \gamma, \alpha, \rho\, \big), \\
T  &= \big(\, \lambda^{(0)}  \ldots, \lambda^{(i-2)},  \gamma, \beta,\rho, \lambda^{(i+2)}, \ldots, \lambda^{(k)}\, \big), \quad\Rightarrow\quad T_i = \big( \gamma, \beta, \rho\, \big).  
\end{align*}
We will see that when $s_i$ or $t_i$ acts on a path indexed by $S$ it stays within the span of the paths indexed by $\Gamma_i(S)$.

For $1 \le i <k$ and  $S \in \mathcal{T}_k(\lambda)$  with $S_i = (\gamma, \alpha, \rho)$ define
\begin{equation}\label{t-action}
t_i\cdot \v_S = \sum_{ \genfrac{}{}{0pt}{}{T\in\Gamma_i(S)}{T_i = (\gamma, \beta, \rho)}} (t)^{\gamma\alpha\rho}_{\gamma\beta\rho} \, \v_T, \hskip.10in
\end{equation}
where the constant $(t)^{\gamma\alpha\rho}_{\gamma\beta\rho}$ is given by 
\begin{equation}
(t)^{\gamma\alpha\rho}_{\gamma\beta\rho} = \delta_{\rho\gamma}
\frac{\sqrt{P_{\alpha}(x-1)P_{\beta}(x-1)}}{P_\gamma(x-1)}
\end{equation}
such that $\delta_{\rho\gamma}$ is the Kronecker delta and $P_\lambda(x)$ is the El-Samra-King polynomial \eqref{ElSamraKing}.  The interesting form of this matrix entry, involving dimensions of orthogonal group modules, is explained in \cite{LR}.  The operator $t$ acts as a projection onto the invariants in $\V \ot \V$ (see  \eqref{eq:stpdef}).  Leduc and Ram use the Schur-Weyl duality and a Markov trace on the centralizer algebra to prove in \cite[Theorem 3.12]{LR} that in a seminormal representation any projection onto the invariants in a tensor product centralizer algebra will have this form.  The fact that $x-1$ is used in place of $x$ is because $\RB_k(n)$ is in Schur-Weyl duality wih $\O(n-1)$.

Similarly, for $S \in \mathcal{T}_k(\lambda)$  with $S_i = (\gamma, \alpha, \rho)$ define
\begin{equation}\label{s-action}
s_i\cdot \v_S = \sum_{ \genfrac{}{}{0pt}{}{T\in\Gamma_i(S)}{T_i = (\gamma, \beta, \rho)}} (s)^{\gamma\alpha\rho}_{\gamma\beta\rho} \, \v_T. \hskip.10in
\end{equation}
The constant $(s)^{\gamma\alpha\rho}_{\gamma\beta\rho}$ is more complicated.  First, we have the cases where $\alpha = \gamma$, $\alpha = \rho$, $\beta = \gamma$, or $\beta = \rho$. These special cases  are given by  the following Kronecker deltas,
\begin{equation}\label{sonrookpaths}
(s)^{\gamma\gamma\rho}_{\gamma\beta\rho} = \delta_{\beta\rho}, \qquad
(s)^{\gamma\rho\rho}_{\gamma\beta\rho} = \delta_{\gamma\beta}, \qquad
(s)^{\gamma\alpha\rho}_{\gamma\gamma\rho} = \delta_{\alpha\rho}, \qquad
(s)^{\gamma\alpha\rho}_{\gamma\rho\rho} = \delta_{\gamma\alpha}.
\end{equation}
In particular, equation \eqref{sonrookpaths} says that if $S$ and $T$ are $i$-compatible and $S_i = (\gamma,\gamma,\rho)$ and $T_i = (\gamma,\rho,\rho)$, then 
$$
s_i\cdot \v_S = \v_T \qquad\hbox{and}\qquad s_i\cdot \v_T = \v_S.
$$
and in the special case where $S_i = (\gamma,\gamma,\gamma)$ we have $s_i \cdot \v_S = \v_S$.

Now, assume that $\alpha \not = \gamma$, $\alpha \not= \rho$, and $\alpha \not=\beta$. Then 
\begin{align}
(s)_{\gamma\alpha\rho}^{\gamma\alpha\rho} &= \begin{cases}
\frac{1}{\diamondsuit\binom{\gamma\alpha\rho}{\gamma\alpha\rho}},  & \text{if $\rho \not=\gamma$}, \\
\frac{1}{\diamondsuit\binom{\gamma\alpha\rho}{\gamma\alpha\rho}} \left(1 - \frac{P_\alpha(x-1)}{P_\gamma(x-1)}\right),  & \text{if $\rho =\gamma$}, \\
\end{cases}\\
(s)_{\gamma\alpha\rho}^{\gamma\beta\rho} &= \begin{cases}
\frac{\sqrt{(\diamondsuit\binom{\gamma\alpha\rho}{\gamma\alpha\rho}-1)(\diamondsuit\binom{\gamma\alpha\rho}{\gamma\alpha\rho}+1)}}{|\diamondsuit\binom{\gamma\alpha\rho}{\gamma\alpha\rho}|},  & \text{if $\rho \not=\gamma$}, \\
\frac{-1}{\diamondsuit\binom{\gamma\alpha\rho}{\gamma\beta\rho}} \frac{\sqrt{P_\alpha(x-1)P_\beta(x-1)}}{P_\gamma(x-1)},  & \text{if $\rho =\gamma$}, \\
\end{cases}
\end{align}
with (and the following definition holds for the special case $\alpha = \beta$),
\begin{equation}
\begin{array}{c}
\diamondsuit\binom{\gamma\alpha\rho}{\gamma\beta\rho} 
\end{array}= \ 
\begin{cases}
\pm\left(ct(b) - ct(a)\right), &  \rho = \beta\pm b \text{ and } \alpha = \gamma\pm a, \\ 
\pm\left(x-2 + ct(b) + ct(a)\right), &  \rho = \beta\pm b \text{ and } \alpha = \gamma\mp a, \\ 
\end{cases}
\end{equation}
where $\lambda = \mu \pm b$ indicates that the partition $\lambda$ is obtained by adding (subtracting) a box $a$ from the partition $\mu$ and where  $ct(a)$ is the \emph{content} of the box $a$, namely
$$
ct(a) = j - i, \qquad \hbox{if $a$ is in row $i$ and column $j$}.
$$

\begin{rem}\label{rem:seminormalrestriction} If a segment of a path consists only of adding boxes, then the action of $s_i$ on that part of the path corresponds  to Young's seminormal representation of the symmetric group $\mathsf{S}_k$ found in \cite[REF]{Yo}.   If the segment consists of adding boxes and doing nothing then the action of $s_i$ and $p_i$ on that part of the path corresponds to the seminormal representation of the rook monoid $\mathsf{R}_k$ found in  \cite[Thm.\ 3.2]{Ha}.  If the segment consists of adding and deleting boxes (but not doing nothing), then the action of $s_i$ and $t_i$ on that part of the path corresponds to the seminormal representation of the Brauer algebra $\mathsf{B}_k(x)$ given in Nazarov \cite[Sec.\ 3]{Nz} and Leduc-Ram \cite[Thm.\ 6.22]{LR}.
\end{rem}

\end{subsection}

\begin{subsection}{Seminormal representations}

For each $\lambda \in \Lambda_k$, extend the action in  \eqref{p-action}, \eqref{t-action},  \eqref{s-action}  linearly to  $\M_k^\lambda$ and 
define an action of $\RB_k(x)$ on $\M_k^\lambda$ by extending the action of the generators to all of $\RB_k(x)$.

\begin{thm} For each $k \ge 0$ and each $\lambda \in \Lambda_k$, the actions of  $p_i, t_i, s_i$ given in \eqref{p-action}, \eqref{t-action},  \eqref{s-action} extend linearly to make $\M_k^\lambda$ an $\RB_k(x)$ module.  
\end{thm}

\begin{proof}  To show that $\M_k^\lambda$ is an $\RB_k(x)$-module, we verify  relations  \eqref{Snrels} ,\eqref{Bnrels}, \eqref{Rnrels},  \eqref{RBnrels}.  Many of these verifications come for free from $\CC\mathsf{S}_n, \mathsf{B}_n(x+1), \CC\mathsf{R}_n$ and  Remark \ref{rem:seminormalrestriction}.   The generators  $p_i, t_i, s_i$ act  locally in the sense that they only care about the parts in positions $i-1,i,i+1$ of the path, so \eqref{Snrels}(b), \eqref{Bnrels}(b), \eqref{Rnrels}(b),(d), and \eqref{RBnrels}(a) all follow from the fact that the generators commute when they are working on non-overlapping parts of  paths.
Furthermore, we  only need to check the relations when we ``do nothing" on one of these steps (that is, $\lambda^{i-1} = \lambda^i$ or $\lambda^i = \lambda^{i+1}$ or both), since the other cases are handled in  $\CC\mathsf{S}_n, \mathsf{B}_n(x)$, and $\CC\mathsf{R}_n$.

The  symmetric group relations \eqref{Snrels} are proven on Brauer type paths in \cite{Nz} and \cite{LR}. These relations are easy to check on paths that do nothing at least once.  As an example, we verify the braid relation \eqref{Snrels}(c) on one such path.  We let $(\delta,\delta,\gamma,\rho)$ denote the path $S = (\lambda^0, \ldots,\lambda^k)$ such that $(\lambda^{i-1},\lambda^{i},\lambda^{i+1},\lambda^{i+2}) = (\delta,\delta,\gamma,\rho)$. Then,
\begin{align*}
s_is_{i+1}s_i(\delta,\delta,\gamma,\rho) & = s_is_{i+1}(\delta,\gamma,\gamma,\rho)  = s_i(\delta,\gamma,\rho,\rho) = \sum_{\alpha} (s)^{\delta\gamma\rho}_{\delta\alpha\rho} (\delta,\alpha,\rho,\rho),\\
s_{i+1}s_is_{i+1}(\delta,\delta,\gamma,\rho) & = \sum_{\alpha} (s)^{\delta\gamma\rho}_{\delta\alpha\rho} s_{i+1}s_i(\delta,\delta,\alpha,\rho)  = \sum_{\alpha} (s)^{\delta\gamma\rho}_{\delta\alpha\rho} s_{i+1}(\delta,\alpha,\alpha,\rho)= \sum_{\alpha} (s)^{\delta\gamma\rho}_{\delta\alpha\rho} (\delta,\alpha,\rho,\rho) .
\end{align*}

In \cite[3.12]{LR} a general formula for action of  $t_i$ is given, since $t$ is projection onto the invariants in $\V \otimes \V$. The Leduc-Ram formula specializes in our case to  \ref{t-action} and the relations involving only  $t_i$'s  follows from the general principles in \cite{LR}.  Most of these relations are also easily verified directly.  One subtlety  is that when $t_i$ acts on a path (even on a path that has no do-nothings) the output is a sum of paths that includes paths that do nothing, and to verify the relations these do-nothing-paths need to be considered separately.   As an example, we verify one of the more complicated relations  \eqref{Bnrels}(e). 
\begin{align*}
s_it_{i+1}t_i(\gamma,\delta,\tau,\rho) &= \delta_{\gamma,\tau} \sum_{\alpha}  t^{\gamma\delta\gamma}_{\gamma\alpha\gamma} s_it_{i+1}(\gamma,\alpha,\gamma,\rho)= \delta_{\gamma,\tau}   t^{\gamma\delta\gamma}_{\gamma\rho\gamma} s_it_{i+1}(\gamma,\rho,\gamma,\rho)\\
&= \delta_{\gamma,\tau} \sum_{\alpha}  t^{\gamma\delta\gamma}_{\gamma\rho\gamma} t^{\rho\gamma\rho}_{\rho\alpha\rho}s_i(\gamma,\rho,\alpha,\rho)
= \delta_{\gamma,\tau} \sum_{\alpha} \sum_{ \genfrac{}{}{0pt}{}{\beta}{\beta\not=\gamma,\beta\not=\alpha}} 
 t^{\gamma\delta\gamma}_{\gamma\rho\gamma} t^{\rho\gamma\rho}_{\rho\alpha\rho} s_{\gamma\beta\alpha}^{\gamma \rho\alpha}(\gamma,\beta,\alpha,\rho) \\
s_{i+1}t_i(\gamma,\delta,\tau,\rho) &= \delta_{\gamma,\tau} \sum_{\beta}  t^{\gamma\delta\gamma}_{\gamma\beta\gamma} s_{i+1}(\gamma,\beta,\gamma,\rho)= 
\delta_{\gamma,\tau} \sum_{\beta}\sum_{ \genfrac{}{}{0pt}{}{\alpha}{\alpha\not=\beta,\alpha\not=\rho}}   t^{\gamma\delta\gamma}_{\gamma\beta\gamma} 
s^{\beta\gamma\rho}_{\beta\alpha\rho}(\gamma,\beta,\alpha,\rho) 
\end{align*}
Note in the first case that $s_{\gamma\beta\alpha}^{\gamma \rho\alpha}=0$ when $\alpha = \rho$, and in the second case $s^{\beta\gamma\rho}_{\beta\alpha\rho}=0$ if $\beta = \gamma$.  So the sum is over all paths that always add or subtract boxes in these components, and thus the two calculations are equal because they are equal in the Brauer algebra \cite{Nz}, \cite{LR}.   It is also possible to prove this directly using a case-by-case examination of the coefficients $t^{\gamma\delta\gamma}_{\gamma\rho\gamma} t^{\rho\gamma\rho}_{\rho\alpha\rho} s_{\gamma\beta\alpha}^{\gamma \rho\alpha}$ and $t^{\gamma\delta\gamma}_{\gamma\beta\gamma} s^{\beta\gamma\rho}_{\beta\alpha\rho}$.

The  relations in \eqref{Rnrels} are straightforward. The most interesting one of them is \eqref{Rnrels}(c):
\begin{align*}
s_i p_i(\gamma,\alpha,\rho) &=  \delta_{\gamma,\alpha} s_i(\gamma,\gamma,\rho)=  \delta_{\gamma,\alpha} (\gamma,\rho,\rho)\\
p_is_i(\gamma,\alpha,\rho) &=  \sum_{\beta} s^{\gamma\alpha\rho}_{\gamma\beta\rho} p_i(\gamma,\beta,\rho)= s^{\gamma\alpha\rho}_{\gamma\gamma\rho} (\gamma,\gamma,\rho)
\end{align*}
which are equal because $s^{\gamma\alpha\rho}_{\gamma\gamma\rho} =\delta_{\alpha,\rho}$ (see \eqref{sonrookpaths}).
The  relations in \eqref{RBnrels} are similarly straightforward, and we verify  \eqref{RBnrels}(c) as an example:
\begin{align*}
t_i p_i t_i(\gamma,\alpha,\rho) &=  \delta_{\gamma,\rho} \sum_{\beta} t^{\gamma\alpha\gamma}_{\gamma\beta\gamma} t_i p_i(\gamma,\beta,\gamma) 
=  \delta_{\gamma,\rho}  t^{\gamma\alpha\gamma}_{\gamma\gamma\gamma} t_i(\gamma,\gamma,\gamma) 
=  \delta_{\gamma,\rho} \sum_{\beta}  t^{\gamma\alpha\gamma}_{\gamma\gamma\gamma} t^{\gamma\gamma\gamma}_{\gamma\beta\gamma}(\gamma,\beta,\gamma).
\end{align*}
Now $ t^{\gamma\alpha\gamma}_{\gamma\gamma\gamma} t^{\gamma\gamma\gamma}_{\gamma\beta\gamma} = 
\frac{\sqrt{P_\alpha(x-1) P_\gamma(x-1)}}{P_\gamma(x-1)}
\frac{\sqrt{P_\gamma(x-1) P_\beta(x-1)}}{P_\gamma(x-1)} = \frac{\sqrt{P_\alpha(x-1) P_\beta(x-1)}}{P_\gamma(x-1)}  = t^{\gamma\alpha\gamma}_{\gamma\beta\gamma}$ and so the displayed sum equals 
$t_i (\gamma,\alpha,\rho)$.   This proves that $\M_k^\lambda$ is an $\RB_k(x)$ module.
\end{proof}

For $\lambda \in \Lambda_k$ let $\Lambda_{k-1}^\lambda = \{\, \mu \in \Lambda_{k-1}\, \vert\,  \mu = \lambda \hbox{ or } \mu = \lambda \pm \square\, \}$ so that $\Lambda_{k-1}^\lambda$ is the set of partitions $\mu\in\Lambda_{k-1}$ that are connected to $\lambda\in\Lambda_{k}$  by an edge in $\mathcal{B}$.   Consider $\RB_{k-1}(x) \subseteq \RB_k(x)$ to be the subalgebra spanned by $p_i$ for $1 \le i \le k-1$ and $t_i, s_i$ for $1 \le i \le k-2$.  Equivalently, $\RB_{k-1}(x)$ is spanned by the diagrams in $\RBc_k$ having a vertical edge connecting $k$th vertex in the top row to the $k$th vertex in the bottom row.   Now order the basis of $\M_k^\lambda$ in such a way that the paths that pass through $\mu \in \Lambda_{k-1}^\lambda$ are grouped together in the ordering, and  let $\rho_k^\lambda$ denote the representation corresponding to the module $\M_k^\lambda$ with respect to this basis. Then since the action of $\RB_{k-1}(x)$ does not see the last edge in the path to $\lambda$, we have the following key property of a seminormal basis,
\begin{equation}\label{seminormalbasis}
\rho_k^\lambda(d) = \bigoplus_{ \genfrac{}{}{0pt}{}{\mu \in \Lambda_{k-1}}{\mu = \lambda \text{ or } \mu = \lambda\pm\square}}  \rho_{k-1}^\mu(d), \qquad\hbox{for all $d \in \RB_{k-1}(x)$}.
\end{equation}
The restriction rules for $\RB_{k-1}(x) \subseteq \RB_k(x)$ follow:
\begin{equation}\label{restrictionrule}
\Res^{\RB_k(x)}_{\RB_{k-1}(x)}(\M_k^\lambda) = \bigoplus_{ \genfrac{}{}{0pt}{}{\mu \in \Lambda_{k-1}}{\mu = \lambda \text{ or } \mu = \lambda\pm\square}}  \M_{k-1}^\mu.
\end{equation}
and \eqref{seminormalbasis} is an explicit realization of this decomposition on the basis level.  

Since the restriction rules \eqref{restrictionrule} distinguish each $\lambda$, we see that $\M_k^\lambda \not\cong \M_k^\gamma$ if $\gamma \not= \lambda$.   Thus, we have constructed a set of nonisomorphic $\RB_k(x)$-modules indexed by $\Lambda_k$ with dimension $\dim(\M_k^\lambda) = m_k^\lambda$.  When $x \in \CC$ is chosen such that $\RB_k(x)$ is semisimple, we know from  \ref{semisimplesection} that these modules have the same index set and the same dimensions as the irreducible $\RB_k(x)$-modules. Thus, they are irreducible, and we have proven the following theorem.

\begin{thm} If $x \in \CC$ is chosen so that $\RB_k(x)$ is semisimple, then $\{\, \M_k^\lambda \, \vert\, \lambda \in \Lambda_k\,\}$ is a complete set of irreducible $\RB_k(x)$-modules. 
\end{thm}

\end{subsection}

\end{section}

\end{document}